\documentclass[11pt]{article}

\usepackage{amsmath}
\usepackage{amsthm}
\usepackage{amsfonts}
\usepackage{amssymb}
\usepackage{enumitem}
\usepackage{array}
\usepackage{setspace}
\usepackage[labelformat=simple]{subcaption}

\usepackage{color}
\usepackage{relsize}
\usepackage{hyperref}

\usepackage{url}
\usepackage{graphicx}
\usepackage[numbers,sort&compress]{natbib}
\usepackage[margin=1.5 in]{geometry}

\newtheorem{theorem}{Theorem}[section]

\newtheorem{lemma}[theorem]{Lemma}
\newtheorem{corollary}[theorem]{Corollary}

\newtheorem{fact}[theorem]{Fact}

\theoremstyle{definition}
\newtheorem{definition}[theorem]{Definition}

\newtheorem{conjecture}[theorem]{Conjecture}

\newtheorem{examplex}{Example}[section]
\newenvironment{example}
{\pushQED{\qed}\examplex}
{\popQED\endexamplex}

\theoremstyle{remark}
\newtheorem{remark}[theorem]{Remark}

\newcommand{\bigboxtimes}{\mathop{\mathlarger{\mathlarger{\boxtimes}}}
\displaylimits}
\newcommand{\super}{\mathcal{U}}

\usepackage{tikz}
\usetikzlibrary{calc}
\usetikzlibrary{decorations.pathreplacing}
\usetikzlibrary{decorations.markings}

\tikzstyle{vertex}=[circle, draw, inner sep=0pt, minimum size=4pt,fill=black]
\newcommand{\vertex}{\node[vertex]}
\tikzstyle{whitevertex}=[circle, draw, inner sep=0pt, minimum size=4pt,fill=white]

\tikzstyle{hollowvertex}=[circle, draw, inner sep=0pt, minimum size=4pt, fill=white]

\tikzstyle{phantomvertex}=[circle, draw, inner sep=0pt, minimum size=4pt,color=white]

\usetikzlibrary{backgrounds}
\usetikzlibrary{arrows.meta}
\tikzset{
  .../.tip={[sep=2pt 2]
    Round Cap[]. Circle[length=0pt 2,sep=2pt] Circle[length=0pt 2,sep=2pt] Circle[length=0pt 2, sep=2pt 2]}}
\tikzset{
  .../.tip={[sep=2pt 2]
    Round Cap[]. Circle[length=0pt 2,sep=2pt] Circle[length=0pt 2,sep=2pt] Circle[length=0pt 2, sep=2pt 2]}}
\tikzset{
      ellipsis/.tip={
Square[length=2pt,sep=0pt,color=white] Circle[length=1pt,sep=0pt,color=black] Square[length=1pt,sep=0pt,color=white]
Circle[length=1pt,sep=0pt,color=black] Square[length=1pt,sep=0pt,color=white]
Circle[length=1pt,sep=0pt,color=black] Square[length=2pt,sep=0pt,color=white]}}
\tikzset{middlearrow/.style n args={2}{
        decoration={markings,
            mark= at position {#2} with {\arrow{#1}} ,
        },
        postaction={decorate}
    }
}

\begin{document}

	\title{The Threshold Strong Dimension of a Graph }
\author{Nadia Benakli, Novi H. Bong, Shonda Dueck (Gosselin),\\ Linda Eroh, Beth Novick, and Ortrud R. Oellermann\thanks{Supported by an NSERC Grant CANADA, Grant number RGPIN-2016-05237}}


\maketitle

\begin{abstract}
Let $G$ be a connected graph and $u,v$ and $w$ vertices of $G$. Then $w$ is said to {\em strongly resolve} $u$ and $v$, if there is either a shortest $u$-$w$ path that contains $v$ or a shortest $v$-$w$ path that contains $u$. A set $W$ of vertices of $G$ is a {\em strong resolving set} if every pair of vertices of $G$ is strongly resolved by some vertex of $W$.  A smallest strong resolving set of a graph is called a {\em strong basis} and its cardinality, denoted $\beta_s(G)$, the {\em strong dimension} of $G$. The {\em threshold strong dimension} of a graph $G$, denoted $\tau_s(G)$, is the smallest strong dimension among all graphs having $G$ as spanning subgraph. A graph whose strong dimension equals its threshold strong dimension is called $\beta_s$-{\em irreducible}. In this paper we establish a geometric characterization for the threshold strong dimension of a graph $G$ that is expressed in terms of the smallest number of paths (each of sufficiently large order) whose strong product admits a certain type of embedding of $G$.  We demonstrate that the threshold strong dimension of a graph is not equal to the previously studied threshold dimension of a graph. Graphs with strong dimension $1$ and $2$ are necessarily $\beta_s$-irreducible. It is well-known that the only graphs with strong dimension $1$ are the paths. We completely describe graphs with strong dimension $2$ in terms of the strong resolving graphs introduced by Oellermann and Peters-Fransen. We obtain sharp upper bounds for the threshold strong dimension of general graphs and determine exact values for this invariant for certain subclasses of trees.\\
{\bf Key words:} strong dimension of graphs, threshold strong dimension, embeddings in strong products of graphs, bounds for the threshold strong dimension, graphs with vertex covering number 2 realizable by strong resolving graphs, threshold strong dimension and trees

\end{abstract}

\section{Introduction}
Motivated by a problem in network security, Slater \cite{Slater1975} initiated the study of the metric dimension of a graph. Let $u,v,$ and $w$  be vertices of a connected graph $G$. Then $w$ is said to {\em resolve} $u$ and $v$, if the distance $d_G(u,w)$ from $u$ to $w$ does not equal the distance $d_G(v,w)$ from $v$ to $w$. If $G$ is clear from context we will write $d(x,y)$ instead of $d_G(x,y)$.  A set $W$ of vertices of $G$ resolves $G$ if every pair of vertices in $G$ is {\em resolved} by some vertex of $W$. A smallest resolving set of a graph is called a {\em metric basis} and its cardinality the {\em metric dimension} of $G$, denoted by $\beta(G)$. Thus, if $W$ is a resolving set for a graph $G$, then the location of an intruder in a network can be uniquely determined if distance detecting devices are placed at each of the vertices in $W$. If $w_1,w_2, \ldots, w_k$ is an ordering of the vertices of $W$, the set of vectors $\{(d(v,w_1), d(v,w_2), \ldots, d(v,w_k)): v \in V(G)\}$ are called the {\em distance vectors} of $G$ relative to the given ordering of the vertices of $W$.

Seb\"{o} and Tannier \cite{SeboTannier2004} observed that there are non-isomorphic graphs $G_1$ and $G_2$ on the same vertex set that share a common metric basis, say $W$, and that have the same distance vectors relative to some ordering of the vertices of $W$. This motivated the introduction of a stronger version of the metric dimension of a graph for which a corresponding basis uniquely determines all adjacencies of the graph. A vertex $w$ is said to {\em strongly resolve} two vertices $u$ and $v$ of a graph $G$ if there is either a shortest $u$-$w$ path that contains $v$ or a shortest $v$-$w$ path that contains $u$ or, equivalently, either the interval between $u$ and $w$ contains $v$ or the interval between $v$ and $w$ contains $u$, where the {\em interval} between two vertices is the collection of all vertices that lie on some shortest path between these vertices.
If every pair of vertices of  $G$ is strongly resolved by a vertex in some set  $W$ of vertices of $G$, then $W$ is a {\em strong resolving set} for $G$. A smallest strong resolving set is called a {\em strong basis} and its cardinality the {\em strong dimension} of $G$, denoted by $\beta_s(G)$. Thus a strong resolving set of a graph is certainly also a resolving set.

It is natural to ask if the number of detecting devices that are required to uniquely determine the location of an intruder in a network could be reduced if additional links between some pairs of nodes are added. Or equivalently one may ask by how much the dimension of a graph can be reduced by adding edges. The question of how the metric dimension of a graph relates to that of its subgraphs had previously been posed, for example, in \cite{Chartrandetal2000} and \cite{Khulleretal1996}. Mol, Murphy and Oellermann in \cite{MolMurphyOellermann2019} introduced the problem of determining the smallest metric dimension among all graphs having a given graph $G$ as spanning subgraph. This minimum is called the {\em threshold dimension} of $G$ and is denoted by $\tau(G)$. Let $\mathcal{U}(G)$ denote that family of graphs having $G$ as spanning subgraph.  If $H \in \mathcal{U}(G)$ is such that $\beta(H) =\tau(G)$, then $H$ is called a \emph{threshold graph} of $G$. Graphs whose metric dimension cannot be lowered by adding edges will be referred to as {\em $\beta$-irreducible}. So $G$ is $\beta$-irreducible if and only if $\beta(G)=\tau(G)$. Graphs that are not $\beta$-irreducible are called \emph{$\beta$-reducible}. The seminal work on  $\beta$-irreducible graphs appears in \cite{MolMurphyOellermann2019_2}. In this paper we introduce and study the analogue of the threshold dimension for the strong dimension of a graph.

\begin{definition}
The {\em threshold strong dimension} of a graph $G$, denoted by $\tau_s(G)$,  is defined as the smallest strong dimension among all graphs having $G$ as spanning subgraph. A graph $H \in \mathcal{U}(G)$ such that $\beta_s(H) =\tau_s(G)$ is called a \emph{strong dimension  threshold graph}.
A graph $G$ is $\beta_s$-\emph{irreducible} if $\beta_s(G)=\tau_s(G)$ and is $\beta_s$-\emph{reducible} otherwise.
\end{definition}

In Section \ref{preliminaries} we introduce some known results and useful tools. In Section \ref{geometricinterpretation} we establish a geometric interpretation for the threshold strong dimension of a graph in terms of certain types of embeddings in strong products of graphs and we show that there are graphs $G$ for which $\tau_s(G) \ne \tau(G)$. Graphs with strong dimension $2$ are $\beta_s$-irreducible. We study their structure in Section \ref{vertexcoveringnumber2}. Bounds for $\tau_s(G)$ for general graphs are obtained in Section \ref{bounds}. We conclude by finding the threshold strong dimension for some special classes of graphs in Section \ref{specialclasses}.

\section{Preliminaries} \label{preliminaries}

\subsection{The strong resolving graph: a tool for finding the strong dimension}

In \cite{OellermannPetersFransen2007} it was shown that the problem of finding the strong dimension of a connected graph can be transformed to a vertex covering problem. We begin by describing this transformation. Let $u$ and $v$ be vertices of a connected graph $G$. The vertex $v$ is said to be {\em maximally distant} from $u$, denoted $v$ MD $u$, if every neighbour of $v$ is no further from $u$ than $v$, i.e., $d(u,x) \le d(u,v)$ for all $x \in N(v)$. If $u$ MD $v$ and $v$ MD $u$, then we say $u$ and $v$ are \emph{mutually maximally distant} and denote this by $u$ MMD $v$. The \emph{strong resolving} graph $G_{SR}$ of $G$ has as its vertex set $V(G)$ and two vertices $u,v$ of $G_{SR}$ are adjacent if and only if $u$ MMD $v$. For any graph $H$, let $\alpha(H)$ denote the vertex covering number of the graph $H$, i.e., the cardinality of a smallest set $S$ of vertices of $H$ such that every edge is incident with a vertex of $S$. The following reduction of the strong dimension problem to the vertex covering problem was given in  \cite{OellermannPetersFransen2007}.

\begin{theorem} {\em \cite{OellermannPetersFransen2007}} \label{G_SR}
If $G$ is a connected graph, then $\beta_s(G) =\alpha(G_{SR})$.
\end{theorem}

\subsection{The threshold dimension of a graph}

For a connected graph $G$, let $\mbox{diam}(G)$ denote the \emph{diameter} of $G$, i.e., the maximum distance between a pair of vertices of $G$.

If $G_1,G_2,\dots, G_k$ are graphs, then their \emph{strong product} is the graph
\[
G_1\boxtimes G_2\boxtimes\cdots\boxtimes G_k=\bigboxtimes_{i=1}^k G_i,
\]
with vertex set $\{(x_1,x_2,\dots,x_k)\colon\ x_i\in V(G_i)\}$, and for which two distinct vertices $x=(x_1,x_2,\dots,x_k)$ and $y=(y_1,y_2,\dots, y_k)$ are adjacent if and only if for every $1\leq i\leq k$, either $x_iy_i\in E(G_i)$ or $x_i=y_i$.  The distance between $x$ and $y$ in $G_1\boxtimes G_2\boxtimes\cdots\boxtimes G_k$ is given by $\max\{d_{G_i}(x_i,y_i)\colon\ 1\leq i\leq k\}.$  For a graph $G$, we let $G^{\boxtimes, k}$ denote the strong product of $k$ copies of $G$, i.e.,
\[
G^{\boxtimes,k}=\displaystyle\bigboxtimes_{i=1}^k G.
\]

Let $G$ and $H$ be graphs.  A map $\varphi:V(G)\rightarrow V(H)$ is called an \emph{embedding} of $G$ in $H$ if it is injective and preserves the edge relation (i.e., if $xy\in E(G)$, then $\varphi(x)\varphi(y)\in E(H)$).

If $G$ is a subgraph of $H$, then we say that $G$ is an \emph{isometric subgraph} of $H$ if $d_G(u,v)=d_H(u,v)$ for all vertices $u,v\in V(G)$.

Recall that $\super(G)$ denotes the set of all graphs that have $G$ as spanning subgraph. For a graph $G$ and a subset $W\subseteq V(G)$, we let $G[W]$ denote the subgraph of $G$ induced by $W$.  For an embedding $\varphi$ of $G$ in $H$, we let $\varphi(G)=H[\varphi(V(G))]$, i.e., $\varphi(G)$ is the subgraph of $H$ induced by the range of $\varphi$.  Thus, the graph $\varphi(G)$ is isomorphic to the graph $G'\in \super(G)$ with vertex set $V(G')=V(G)$ and edge set $E(G')=\{xy\colon\ \varphi(x)\varphi(y)\in E(\varphi(G))\}.$

We next describe the geometric interpretation of the threshold dimension of a graph developed in \cite{MolMurphyOellermann2019}. To do this, we let $V(P_n)=\{0,\ldots,n-1\}$. Thus, the vertices of $P_n^{\boxtimes, k}$ are $k$-tuples over the set $\{0, \ldots, n-1\}$.  With this choice of notation for the vertex set of $P_n$,  distances in $P_n^{\boxtimes,k}$ can easily be computed.

\begin{fact} \label{distanceinstrongproducts}
If $x=(x_1, \ldots, x_k)$ and $y=(y_1, \ldots, y_k)$ are in $V\left(P_n^{\boxtimes, k}\right)$, then
\[
d(x,y)=\mathrm{max}\{|x_i-y_i|\colon\ 1 \leq i \leq k\}.
\]
In particular, if $x$ and $y$ are distinct, then they are adjacent if and only if $|x_i-y_i|\leq 1$ for every $1 \leq i \leq k$.
\end{fact}

The choice of the vertex labels in $V(P_n)$ is important since they correspond to distances, and thus the labels of the vertices of $P_n^{\boxtimes,k}$ will correspond to vectors of distances.  Let $G$ be a connected graph with resolving set $W=\{w_1,w_2,\dots,w_k\}.$  Then every vertex $x\in V(G)$ is uniquely determined by its vector of distances to vertices in $W$, given  by $\left(d_G(x,w_1),d_G(x,w_2),\dots,d_G(x,w_k)\right)$.  It was shown in \cite{MolMurphyOellermann2019} that the map which takes every vertex $x$ to this vector of distances to $W$ is an embedding of $G$ in $P^{\boxtimes,k}$ for some path $P$.

It was also shown in \cite{MolMurphyOellermann2019} that if $W=\{w_1,w_2,\ldots,w_k\}$ is a resolving set for some graph in $\super(G)$, then there is an embedding $\varphi$ of $G$ in $P^{\boxtimes,k}$ for some path $P$, such that for every vertex $x\in V(G)$, the label of $\varphi(x)$ is exactly the vector of distances in $\varphi(G)$ from $\varphi(x)$ to the vertices of $\varphi(W)$.  More formally these embeddings are defined as follows:

\begin{definition}
Let $G$ be a graph, let $W=\{w_1,w_2,\dots,w_k\}$ be a subset of $V(G)$, and let $P$ be a path.  A \emph{$W$-resolved embedding} of $G$ in $P^{\boxtimes,k}$ is an embedding $\varphi$ of $G$ in $P^{\boxtimes,k}$ such that for every $x\in V(G)$, we have
\[
\varphi(x)=\left(d_{\varphi(G)}(\varphi(x),\varphi(w_1)),
\dots,d_{\varphi(G)}(\varphi(x),\varphi(w_k))\right),
\]
i.e., for every $1\leq i\leq k$, the $i$th coordinate of $\varphi(x)$ is exactly the distance between $\varphi(w_i)$ and $\varphi(x)$ in $\varphi(G)$.
\end{definition}

The geometric interpretation of the threshold dimension of a graph given in \cite{MolMurphyOellermann2019} is summarized in the following two results.

\begin{theorem} {\em \cite{MolMurphyOellermann2019}} \label{Correspondence}
Let $G$ be a connected graph of diameter $D$, and let $W=\{w_1,w_2,\dots,$ $w_k\}\subseteq V(G)$.  Then $W$ is a resolving set for some graph $H\in \super(G)$ if and only if there is a $W$-resolved embedding of $G$ in $P_{D+1}^{\boxtimes,k}$.
\end{theorem}

The following consequence of this theorem gives a  geometric interpretation for the threshold dimension.

\begin{corollary} {\em \cite{MolMurphyOellermann2019}} \label{ThresholdEmbedding}
Let $G$ be a connected graph of diameter $D$. Then $\tau(G)$ is the minimum cardinality of a set $W\subseteq V(G)$ such that there is a $W$-resolved embedding of $G$ in $P^{\boxtimes,|W|}_{D+1}$.
\end{corollary}

\section{A geometric interpretation for the threshold strong dimension}  \label{geometricinterpretation}

As described in Section \ref{preliminaries}, a characterization for the threshold dimension of a graph with a geometric flavour was established in  \cite{MolMurphyOellermann2019}. In this section we establish a characterization for the threshold strong dimension of a graph that has a geometric flavour and builds on the geometric-type characterization of the threshold dimension given in \cite{MolMurphyOellermann2019}. We  also demonstrate that the threshold strong dimension of a graph  may not equal the threshold dimension.

\subsection{A geometric characterization for the threshold strong dimension}

We show next that with one additional condition the $W$-resolved embeddings described in Section 2, give rise to a geometric interpretation of the threshold strong dimension. We begin with a useful lemma.

\begin{lemma} \label{equality}
Let $G$ be a connected graph with diameter $D$ and let $W=\{w_1, w_2, \ldots, w_k\}$ be a set of vertices of $G$.  If $\varphi(G)$ is a $W$-resolved embedding of $G$ in $P_{D+1}^{\boxtimes, k}$  and $x \in V(G)$, then $$d_{\varphi(G)}(\varphi(x),\varphi(w_i))= d_{P_{D+1}^{\boxtimes, k}}(\varphi(x),\varphi(w_i)).$$
\end{lemma}
\begin{proof}
Since $\varphi(G)$ is a $W$-resolved embedding of $G$, we know, by definition, that
\[
\varphi(x)=\left(d_{\varphi(G)}(\varphi(x),\varphi(w_1)),
\dots,d_{\varphi(G)}(\varphi(x),\varphi(w_k))\right)
\]
and, in particular, 

\[
\varphi(w_j)=\left(d_{\varphi(G)}(\varphi(w_j),\varphi(w_1)), \ldots ,d_{\varphi(G)}(\varphi(w_j),\varphi(w_j)),
\dots,d_{\varphi(G)}(\varphi(w_j),\varphi(w_k))\right).
\]
By Fact \ref{distanceinstrongproducts}, 

\begin{align*}
d_{P_{D+1}^{\boxtimes, k}}(\varphi(w_j), \varphi(x))& = \mathrm{max}\{|d_{\varphi(G)}(\varphi(w_i),\varphi(x))- d_{\varphi(G)}(\varphi(w_i),\varphi(w_j))|\colon\ 1 \leq i \leq k\}\\
& \ge |d_{\varphi(G)}(\varphi(w_j),\varphi(x))- d_{\varphi(G)}(\varphi(w_j),\varphi(w_j))|\\
&=d_{\varphi(G)}(\varphi(w_j),\varphi(x)).
\end{align*}

\noindent Since $\varphi(G)$ is a subgraph of $P_{D+1}^{\boxtimes, k}$,
\[d_{P_{D+1}^{\boxtimes, k}}(\varphi(w_j), \varphi(x)) \le d_{\varphi(G)}(\varphi(w_j), \varphi(x)).\]

\noindent The result now follows.
\end{proof}

\begin{theorem}\label{strong_resolving_characterization}
Let $G$ be a connected graph of diameter $D$, and let $W=\{w_1,w_2,\dots,$ $w_k\}\subseteq V(G)$.  Then $W$ is a strong resolving set for some graph $H\in \super(G)$ if and only if there is a $W$-resolved embedding $\varphi(G)$ of $G$ in $P_{D+1}^{\boxtimes,k}$ such that $\varphi(G)$ is an isometric subgraph of $P_{D+1}^{\boxtimes, k}$.
\end{theorem}
\begin{proof}
Suppose $W$ is a strong resolving set for some graph $H\in \super(G)$. Since $W$ is a strong resolving set for $H$, it is also a resolving set for $H$. Since the diameter$(H) \leq $ diameter$(G)=D$, we know from Theorem \ref{Correspondence} that there is a $W$-resolved embedding $\varphi(H)$ of $H$ in $P_{D+1}^{\boxtimes, k}$ such that if $v \in V(H)$, then  
\begin{align*}
\varphi(v)&=(d_{\varphi(H)}(\varphi(v),\varphi(w_1)), d_{\varphi(H)}(\varphi(v),\varphi(w_2)), \ldots, d_{\varphi(H)}(\varphi(v),\varphi(w_k)))\\
&=(d_H(v,w_1), d_H(v,w_2), \ldots, d_H(v,w_k)).
\end{align*}

\bigskip

By Lemma \ref{equality} the $j^{th}$ coordinate of $\varphi(v)$ is the distance from $\varphi(v)$ to $\varphi(w_j)$ in $P_{D+1}^{\boxtimes, k}$. So $H$, when viewed as a subgraph of $P_{D+1}^{\boxtimes, k}$,  preserves the distances between every vertex $v$ of $H$ and every vertex $w_j \in W$. 

Let $a,b$ be two vertices of $V(H)-W$. We now show that the distance between $a, b$ in (the embedding of) $H$  equals the distance between  $a$ and $b$ in $P_{D+1}^{\boxtimes, k}$.  Since $\varphi(H)$ is an embedding of $H$ in $P_{D+1}^{\boxtimes, k}$, 
\[d_H(a,b) \ge d_{\varphi(H)}(\varphi(a), \varphi(b)) \ge d_{P_{D+1}^{\boxtimes, k}}(\varphi(a), \varphi(b)).\]

Since $W$ strongly resolves $H$, there is some $w_j \in W$ such that either the interval between $a$ and $w_j$ in $H$ contains $b$ or the interval between $b$ and $w_j$ in $H$ contains $a$. We may assume the former occurs. We have already observed that $d_H(w_j,a) = d_{P_{D+1}^{\boxtimes, k}}(\varphi(w_j), \varphi(a))$ and $d_H((w_j,b) = d_{P_{D+1}^{\boxtimes, k}}(\varphi(w_j), \varphi(b))$. Since 

\begin{align*}
d_H(a,b)&= d_H(w_j,a)-d_H(w_j,b) \\
&=d_{P_{D+1}^{\boxtimes, k}}(\varphi(w_j), \varphi(a))-d_{P_{D+1}^{\boxtimes, k}}(\varphi(w_j), \varphi(b))\\
&\le d_{P_{D+1}^{\boxtimes, k}}(\varphi(a), \varphi(b))
\end{align*}
we see that $H$, when viewed as a subgraph of $P_{D+1}^{\boxtimes, k}$, preserves distances between every pair of vertices $a,b \in V(H)-W$. So $H$ is an isometric subgraph of $P_{D+1}^{\boxtimes, k}$.

For the converse suppose that there is a $W$-resolved embedding $\varphi(G)$ of $G$ in $P_{D+1}^{\boxtimes,k}$ such that $\varphi(G)$ is an isometric subgraph of $P_{D+1}^{\boxtimes, k}$. From Theorem 2.4,  $W$ is a resolving set. We show that $W$ is in fact,  a strong resolving set for $\varphi(G)$.

Again,  let $a,b$ be two vertices of $V(H)-W$.
\[
\varphi(a)=\left(d_{\varphi(G)}(\varphi(a),\varphi(w_1)),
\dots,d_{\varphi(G)}(\varphi(a),\varphi(w_k))\right),
\]
and 
\[
\varphi(b)=\left(d_{\varphi(G)}(\varphi(b),\varphi(w_1)),
\dots,d_{\varphi(G)}(\varphi(b),\varphi(w_k))\right).
\]
Moreover, since $\varphi(G)$ is an isometric subgraph of $P_{D+1}^{\boxtimes, k}$ and by Fact \ref{distanceinstrongproducts}, we have

\begin{align*}
d_{\varphi(G)}(\varphi(a), \varphi(b))&=d_{P_{D+1}^{\boxtimes, k}}(\varphi(a), \varphi(b))\\
&=\mathrm{max}\{|d_{P_{D+1}^{\boxtimes, k}}(\varphi(a),\varphi(w_j)) - d_{P_{D+1}^{\boxtimes, k}}(\varphi(b),\varphi(w_j)) | \colon\ 1 \leq j \leq k\}\\
&=\mathrm{max}\{|d_{\varphi(G)}(\varphi(a),\varphi(w_j))-d_{\varphi(G)}(\varphi(b),\varphi(w_j)) | \colon\ 1 \leq j \leq k\}.
\end{align*}

\noindent Let $i \in \{1,2, \ldots k\}$ be such that 
\[d_{\varphi(G)}(\varphi(a), \varphi(b))=|d_{\varphi(G)}(\varphi(a),\varphi(w_i))-d_{\varphi(G)}(\varphi(b),\varphi(w_i)) | \]
and assume wlog that
\[d_{\varphi(G)}(\varphi(a),\varphi(w_i)) \ge d_{\varphi(G)}(\varphi(b)\varphi(w_i)). \]

Then 
\[d_{\varphi(G)}(\varphi(a),\varphi(w_i))=d_{\varphi(G)}(\varphi(a), \varphi(b)) + d_{\varphi(G)}(\varphi(b),\varphi(w_i)).\]

Thus by taking in $\varphi(G)$ a shortest path from $\varphi(a)$ to $\varphi(b)$ followed by a shortest path from $\varphi(b)$ to $\varphi(w_i)$ we obtain a shortest $\varphi(a)$--$\varphi(w_i)$ path in $\varphi(G)$ that contains $\varphi(b)$. Hence $\{\varphi(w_1), \varphi(w_2), \ldots, \varphi(w_k)\}$ is a strong resolving set for $\varphi(G)$.
Since $H =\varphi(G) \in \super(G)$, this completes the proof of the converse. 

\end{proof}

\noindent As a consequence of this theorem we have the following.

\begin{corollary}\label{ThresholdStrongEmbedding}
Let $G$ be a connected graph of diameter $D$. Then $\tau_s(G)$ is the minimum cardinality of a set $W \subseteq V(G)$ for which there is a $W$-resolved embedding $\varphi(G)$ of $G$ in $P^{\boxtimes,|W|}_{D+1}$  that is an isometric subgraph of $P^{\boxtimes,|W|}_{D+1}$.
\end{corollary}

\subsection{Comparing the threshold strong dimension with the threshold dimension and the strong isometric dimension}
In this section we show that the threshold strong dimension does not equal either the threshold dimension or the strong isometric dimension.
\subsubsection{The threshold strong dimension and the threshold dimension are not equal} 
We show that the threshold strong dimension does not equal the threshold dimension  by exhibiting a specific graph whose threshold dimension is $2$,   but whose threshold strong dimension exceeds $2$.   

Let $G$ be a graph with metric dimension $2$, metric basis $W=\{w_1,  w_2\}$ and diameter $D$.   Then $G$ is not a path and hence $\tau(G)=2$.   It follows from Theorem \ref{Correspondence} and Corollary \ref{ThresholdEmbedding}  that $G$ has a $W$-resolved embedding in $P^{\boxtimes,2}_{D+1}$.   By Theorem 3.2 and Corollary 3.3,  $W=\{w_1,w_2\}$ is a strong basis of some graph $H\in \super(G)$ if and only if there is a $W$-resolved isometric embedding of $G$ in $P^{\boxtimes,2}_{D+1}$.

 \begin{example}\label{OExample}
Let $G$ be the graph shown in black in Figure \ref{fig:smallOExample} as an embedding in  $P^{\boxtimes,2}_{D+1}$.  By  Theorem \ref{Correspondence} the set $\{w_1,w_2$\} is a metric basis,  but not a strong basis, since for example, $c_5$ and $f_3$ are not strongly resolved by either $w_1$ or $w_2$.  
\end{example}

\bigskip
\begin{figure}[h]
\centering
\begin{tikzpicture}[scale=.8,every node/.style={draw,shape=circle,outer sep=2pt,inner sep=1pt,minimum
	size=.2cm}]
\foreach \x in {-3,...,3}
\foreach \y in {-3,...,3}
{
	\vertex[opacity=0.25]  (\x\y) at (\x,\y) {};
}
\foreach \x in {-3,...,2}
\foreach \y in {-3,...,2}
{
	\pgfmathtruncatemacro{\a}{\x+1}
	\pgfmathtruncatemacro{\b}{\y+1}
	\path[opacity=0.25]
	(\x\y) edge (\a\b)
	(\x\b) edge (\a\y);
}
\foreach \x in {-3,...,2}
\foreach \y in {-3,...,3}
{
	\pgfmathtruncatemacro{\a}{\x+1}
	\path[opacity=0.25]
	(\x\y) edge (\a\y)
	(\y\x) edge (\y\a);
}

\foreach \k in {1,...,3}{
	\foreach \j in {1,...,5}{
		\node[fill=black]  (1\j\k) at (-5+\j+\k,3-\j) {};
	}
}

\node[fill=black]  (161) at (2,-3) {};

\foreach \k in {3,...,5}{
	\node[fill=black]  (10\k) at (-5+\k,3) {};
}
\node[fill=black]  (114) at (0,2) {};

\foreach \k in {4, ..., 5}{	
	\node[fill=black]  (13\k) at (-2+\k,0) {};
}
\node[fill=black]  (144) at (3,-1) {};

%
\node[draw=none] at (-3.3, 1.75){\tiny{$w_1$}};

\node[draw=none] at (-1.6, 3.15){\tiny{$a_3$}};
\node[draw=none] at (-1.6, 2.15){\tiny{$a_2$}};
\node[draw=none] at (-2.3, 0.75){\tiny{$a_1$}};

\node[draw=none] at (-0.6, 3.15){\tiny{$b_4$}};
\node[draw=none] at (-0.6, 2.15){\tiny{$b_3$}};
\node[draw=none] at (-0.6, 1.15){\tiny{$b_2$}};
\node[draw=none] at (-1.3,-0.25){\tiny{$b_1$}};

\node[draw=none] at (0.4, 3.15){\tiny{$c_5$}};
\node[draw=none] at (0.4, 2.15){\tiny{$c_4$}};
\node[draw=none] at (0.4, 1.15){\tiny{$c_3$}};
\node[draw=none] at (0.4, 0.15){\tiny{$c_2$}};
\node[draw=none] at (-0.3,-1.25){\tiny{$c_1$}};

\node[draw=none] at (1.4, 0.15){\tiny{$d_3$}};
\node[draw=none] at (1.4, -0.85){\tiny{$d_2$}};
\node[draw=none] at (0.7, -2.25){\tiny{$d_1$}};

\node[draw=none] at (2.4, 0.15){\tiny{$e_3$}};
\node[draw=none] at (2.4, -0.85){\tiny{$e_2$}};
\node[draw=none] at (2.4, -1.85){\tiny{$e_1$}};
\node[draw=none] at (1.7, -3.25){\tiny{$w_2$}};

\node[draw=none] at (3.4, 0.15){\tiny{$f_3$}};
\node[draw=none] at (3.4, -0.85){\tiny{$f_2$}};
\node[draw=none] at (3.4, -1.85){\tiny{$f_1$}};

%
\foreach \k in {1,...,3}{
	\foreach \j in {1,...,4}{
		\pgfmathtruncatemacro{\next}{1+\j}
		\draw[thick] (1\j\k) to node[draw=none] {} 	(1\next\k);
	}
}
\draw[thick] (103) to node[draw=none] {} 	(113);
\draw[thick] (104) to node[draw=none] {} 	(114);
\draw[thick] (134) to node[draw=none] {} 	(144);
\draw[thick] (151) to node[draw=none] {} 	(161);

\draw[thick] (103) to node[draw=none] {}(105);
\draw[thick] (111) to node[draw=none] {}(114);	
\draw[thick] (121) to node[draw=none] {}(123);
\draw[thick] (131) to node[draw=none] {}(135);
\draw[thick] (141) to node[draw=none] {}(144);
\draw[thick] (151) to node[draw=none] {}(153);


\draw[thick] (103) to node[draw=none] {} 	(121);
\draw[thick] (104) to node[draw=none] {} 	(131);
\draw[thick] (105) to node[draw=none] {} 	(141);
\draw[thick] (133) to node[draw=none] {} 	(151);
\draw[thick] (134) to node[draw=none] {} 	(161);
\draw[thick] (135) to node[draw=none] {} 	(153);


\draw[thick] (111) to node[draw=none] {} 	(103);
\draw[thick] (112) to node[draw=none] {} 	(104);
\draw[thick] (121) to node[draw=none] {} 	(105);
\draw[thick] (122) to node[draw=none] {} 	(114);	
\draw[thick] (131) to node[draw=none] {} 	(123);
\draw[thick] (141) to node[draw=none] {} 	(133);
\draw[thick] (142) to node[draw=none] {} 	(134);
\draw[thick] (151) to node[draw=none] {} 	(135);
\draw[thick] (152) to node[draw=none] {} 	(144);
\draw[thick] (161) to node[draw=none] {} 	(153);

\end{tikzpicture}
\caption{Example 3.1}
\label{fig:smallOExample}
\end{figure}
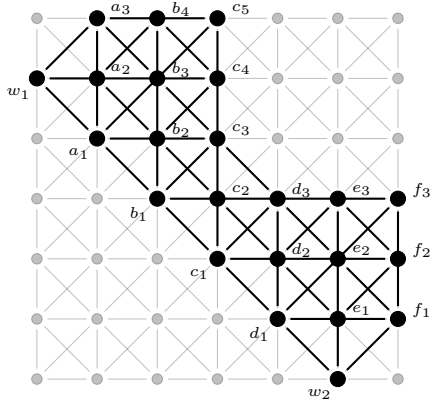


We now show that there is no  $H \in \super(G)$ such that $\beta_s(H)=2$. We begin by establishing some useful lemmas. 
For the first of these, we will assume that the vertices of $P_{D+1}^{\boxtimes, 2}$ have been labeled as in Theorem \ref{Correspondence} and that this graph has been drawn in $\mathbb{R}^2$ so that its vertices are positioned at the points of $\mathbb{R}^2$ that correspond to its vertex labels.

\begin{lemma} \label{boundaries_for_embedding}
Let $G$ be a connected graph with diameter $D$, metric dimension $2$ and metric basis $W=\{w_1,w_2\}$. Let $a=d_G(w_1,w_2)$ and let $\varphi$ be the $W$-resolved embedding described in Theorem \ref{Correspondence}. Then $\varphi(G)$ is contained in the subgraph of $P_{D+1}^{\boxtimes, 2}$ bounded by the paths  
\begin{align*}
Q_1:& (0,a), (1,a+1), \ldots, (D-a, D),\\
Q_2:& (0,a),(1,a-1), \ldots, (a,0),\\
Q_3:&(a,0), (a+1,1), \ldots, (D, D-a),\\
Q_4: &(D-a,D), (D-a+1, D), \ldots, (D,D),~\mbox{and}\\
Q_5: &(D, D-a), (D,D-a+1), \ldots, (D,D)
\end{align*}
\end{lemma}
\begin{proof}
First note that if $(x,y)$ is a vertex in  $\phi(G)$, then $x,y\le D$. To complete the proof of the lemma we will show that  $\varphi(G)$ contains no vertices with labels $(x,y)$ where either
$x+y <a$, or $y<x-a$,  or  $x<y-a$.
 Assume, to the contrary that,  $\varphi(v)=(x,y)$ is a vertex of $\varphi(G)$,  where either 
\begin{itemize}
\item $x+y=d_{\varphi(G)}(w_1,v) + d_{\varphi(G)}(v,w_2) < a = d_{\varphi(G)}(w_1,w_2)$,  or  
\item $y=d_{\varphi(G)}(v,w_2) < x-a = d_{\varphi(G)}(v,w_1) - d_{\varphi(G)}(w_1,w_2)$,  or
\item  $x=d_{\varphi(G)}(v,w_1) < y-a=  d_{\varphi(G)}(v,w_2) - d_{\varphi(G)}(w_1,w_2)$.
\end{itemize}  
In each case we see that  the triangle inequality of the distance metric in graphs is violated, thereby completing the proof.
 
\end{proof}

The subgraph of $P_{D+1}^{\boxtimes, 2}$ bounded by the paths $Q_1, Q_2, \ldots, Q_5$ of Lemma \ref{boundaries_for_embedding} will be referred to as the {\em feasible region}. For a vertex $x$ of a connected graph $G$ and an integer $i \ge 0$, let $N_i(x,G)$ or $N_i(x)$, if $G$ is clear from context, be the set of all vertices distance $i$ from $x$ in $G$.  Let $e(x)$ denote the {\em eccentricity} of $x$, i.e., the distance from $x$ to the furthest vertex from $x$ in $G$. For vertices $x$, $y$ in a connected graph a shortest $x$-$y$ path will be referred to as an $x$-$y$ {\em geodesic}. Using Lemma \ref{boundaries_for_embedding} we obtain new proofs (with a geometric flavour) for a  set of useful properties of  graphs with metric dimension $2$ that were established in \cite{Sudhakaraetal2009}.

\begin{lemma}\label{dim2Properties}
Let $G$ be a connected graph of metric dimension $2$ and let $W=\{w_1,w_2\}$ be a metric basis for $G$.    
Then

\begin{enumerate}
\item  The degree of $w_j$ is at most $3$ for $j=1,2$.
\item There is a unique shortest $w_1$-$w_2$ path in $G$,  and every vertex on that path has degree at most $5$. 
\item  The subgraph induced by $N_i(w_j)$ is  a union of paths and $|N_i(w_j)| \le 2i+1$ for $0 \le i \le e(w_j)$. 
\item For any $v\in N_i(w_j)$,  $v$ is adjacent to at most three vertices in $N_{i+1}(w_j)$,  for $0\le i \le e(w_j)-1$.  Similarly,  there are at most three vertices in $N_{i-1}(v_j)$ adjacent to $v$ for  $ 1 \le i \le e(w_j)$.  
\end{enumerate}
\end{lemma}

\begin{proof} Let $G$ and $W$ be as stated and let $\varphi$ be a $W$-resolved embedding of $G$ into $P^{\boxtimes,2}_{D+1}$,  where $D=\mbox{diam}(G)$.   Let $\varphi$ be the $W$-resolved embedding of $G$ described in Lemma \ref{boundaries_for_embedding}. Then the coordinates of $\varphi(w_1)$ and $\varphi(w_2)$ are $(0,a)$ and $(a,0)$, respectively,  where $a=d_{\varphi(G)}(\varphi(w_1), \varphi(w_2))=d_G(w_1,w_2)$.    Property 1 now follows immediately from the fact that in $P^{\boxtimes,2}_n$,   the vertex $\phi(w_1)= (0,a)$ is adjacent to exactly three vertices in the feasible region,  namely $(1,a+1)$,  $(1,a)$,  and $(1, a-1)$.   A similar analysis holds  for $\phi(w_2)= (a,0)$.  To see that there is a unique $w_1$-$w_2$ geodesic in $G$, we need only note that the path
$$(0,a),(1, a-1),(2, a-2), \cdots, (a,0)$$
is the unique $\varphi(w_1)$-$\varphi(w_2)$ geodesic in $\varphi(G)$: Indeed,  because $\varphi$ is $W$-resolving,  the vertices in this path must be the images of vertices which induce a $w_1$-$w_2$ geodesic in $G$.  Furthermore,  for each $v$ on this path,  exactly five of its neighbours in $P^{\boxtimes,2}_{D+1}$,   are in the feasible  region.   This establishes Property 2.  Since the vertices of
$\varphi(N_k(w_1)) $  are distance $k$ from $\varphi(w_1)$ in $P^{\boxtimes,2}_{D+1}$  they form, by Lemma \ref{boundaries_for_embedding} a subset of $\{ (k,a-k), (k, a-k+1), \ldots , (k,a+k)   \}$.   A similar situation holds for $w_2$. Since  $\{ (k,a-k), (k, a-k+1), \ldots , (k,a+k)   \}$ induces a path in $P^{\boxtimes,2}_{D+1}$,  Property 3 follows.   Lastly,
note that  a vertex $v=(x,y)$ in $\varphi(N_i(w_j))$ for $j=1,2$ has at most three neighbours in  $\varphi(N_{i-1}(w_j))$,  namely $(x-1,y+1)$, $(x-1, y)$,  and $(x-1, y-1)$,  and at most three neighbours in $\varphi(N_{k+1}(w_j))$,  namely $(x+1, y-1)$,  $(x+1, y)$,  and $(x+1, y+1)$,  thereby establishing Property 4.

\end{proof}

If $G$ is a graph with metric dimension 2 and basis $W=\{w_1,w_2\}$, then, by Lemma \ref{dim2Properties} there is a unique $w_1$-$w_2$ geodesic in $G$ to which we will refer as the {\it  diagonal}.  Every vertex of $G$ on this path is called a {\it diagonal vertex}, and any other vertex is referred to as a {\it non-diagonal}.  We are now prepared to show that when $G$ is the graph of Example \ref{OExample}, $\tau_s(G) >2$.

\begin{theorem} \label{threshold_not_strongthreshold}
Let $G$ be the graph of Example \ref{OExample}. Then $\tau_s(G) >2$. 
\end{theorem}
\begin{proof}
Assume to the contrary, that there is a $H \in \super(G)$ such that $\beta_s(H)=2$. Let $W=\{w_1,w_2\}$. We claim that  $W=\{w_1,w_2\}$ is the unique basis for $H$. By Lemma \ref{dim2Properties}(1) and (3) we know that if $H$ is a graph with metric dimension $2$, then the vertices in the basis have degree at most $3$ and the subgraphs induced by their neighbourhoods are acyclic. Hence the only candidates for basis vertices are $w_1$ and $w_2$. This proves our claim. 

Moreover, by definition of a $W$-resolved embedding, the distance in $H$ from any vertex $v$ to either $w_1$ or $w_2$  equals the distance from $v$ to $w_1$ or $w_2$ in $G$, respectively. In particular, $d_H(w_1,w_2)=d_G(w_1,w_2)=5$. Thus $w_1a_1b_1c_1d_1w_2$ is the $w_1$-$w_2$ diagonal, i.e., the unique $w_1$-$w_2$ geodesic in $H$.

By Lemma \ref{dim2Properties} (2) the degree in $H$ of each interior vertex on this diagonal is at most $5$. Hence the neighbourhood of each of the vertices $a_1,b_1,c_1,d_1$ in $H$ is the same as the respective neighbourhood in $G$. Thus, in a $W$-resolved embedding of $H$ in $P^{\boxtimes,2}_{7}$ the diagonal vertices of $H$ appear in exactly the positions shown in Figure \ref{fig:smallOExample}. Moreover, the neighbours of these vertices in $H$ are precisely the same as their neighbours in $G$, by Lemma \ref{dim2Properties} (1) and (2). Moreover, by  Lemma \ref{boundaries_for_embedding}, we see that in the $W$-resolved embedding of $H$ in $P^{\boxtimes,2}_{7}$ these neighbours of the diagonal vertices necessarily appear in the positions shown in Figure \ref{fig:smallOExample}. The positions of the remaining vertices of $H$ in the $W$-resolved embedding of $H$ in $P^{\boxtimes,2}_{7}$ are now forced to coincide with their positions shown in Figure \ref{fig:smallOExample}. Hence $H$ necessarily has the same edges as $G$. However, then $W$ does not strongly resolve $H$, a contradiction.

\end{proof}

Indeed we believe that the difference $\tau_s(G) -\tau(G)$ can be arbitrarily large. To this end let $G_1$ be the graph shown in Figure \ref{fig:smallOExample}. Let $G_2$ be the graph obtained from two copies $G^1_1$ and $G^2_1$ of the graph $G_1$  by identifying the vertices corresponding to $w_2$ and $f_1$ in $G^1_1$ with the vertices $w_1$ and $a_3$, respectively in $G^2_1$ and adding the edge between the vertex $f_2$ in $G^1_1$ and the vertex $b_4$ from $G^2_1$, as well as the edge between the vertex $e_1$ in $G^1_1$ and the vertex $a_2$ in $G^2_1$.  The  graph $G_2$ is shown in Figure \ref{fig:OExample}. In general for $n \ge 2$, let $G_n$ be the graph obtained from $n$ copies $G^1_1,G^2_1, \ldots, G^n_1$ of $G_1$ by identifying for each $1 \le i <n$ the vertices labeled $w_2$ and $f_1$ in $G^{i}_1$ with the vertices labeled $w_1$ and $a_3$ in $G^{i+1}_1$ and then adding the edge between the vertex $f_2$ in $G^{i}_1$ and the vertex $b_4$ in $G^{i+1}_1$, as well as adding the edge between the vertex $e_1$ in $G^{i}_1$ and the vertex $a_2$ in $G^{i+1}_1$. It is readily seen that $\tau(G_n)=2$. Using an exhaustive computer search it was shown that $\tau_s(G_2) =4$. We conjecture the following:

\begin{conjecture}\label{gap}
For every positive integer $k$,  there is a positive integers $n$ such that 
\[\tau_s(G_{n}) \ge \tau (G_n) + k.\]
\end{conjecture}

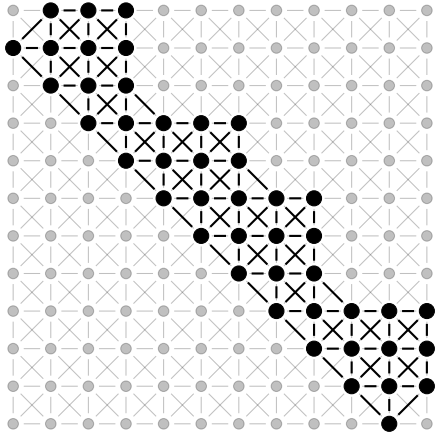
\begin{figure}[h]
\centering
\begin{tikzpicture}[scale=.5,every node/.style={draw,shape=circle,outer sep=2pt,inner sep=1pt,minimum
	size=.2cm}]
\foreach \x in {-5,...,6}
\foreach \y in {-5,...,6}
{
	\vertex[opacity=0.25]  (\x\y) at (\x,\y) {};
}
\foreach \x in {-5,...,5}
\foreach \y in {-5,...,5}
{
	\pgfmathtruncatemacro{\a}{\x+1}
	\pgfmathtruncatemacro{\b}{\y+1}
	\path[opacity=0.25]
	(\x\y) edge (\a\b)
	(\x\b) edge (\a\y);
}
\foreach \x in {-5,...,5}
\foreach \y in {-5,...,6}
{
	\pgfmathtruncatemacro{\a}{\x+1}
	\path[opacity=0.25]
	(\x\y) edge (\a\y)
	(\y\x) edge (\y\a);
}

\foreach \k in {1,...,3}{
	\foreach \j in {1,...,10}{
		\node[fill=black]  (1\j\k) at (-6+\j+\k-1,6-\j) {};
	}
}

\node[fill=black]  (1111) at (5,-5) {};

\foreach \k in {3,...,5}{
	\node[fill=black]  (10\k) at (-6+\k-1,6) {};
}
\node[fill=black]  (114) at (-2,5) {};

\foreach \k in {4, ..., 5}{	
	\node[fill=black]  (13\k) at (-3+\k-1,3) {};
}
\node[fill=black]  (144) at (1,2) {};

\foreach \k in {4, ..., 5}{	
	\node[fill=black]  (15\k) at (\k-2,1) {};
}
\node[fill=black]  (164) at (3,0) {};

\foreach \k in {4, ..., 5}{	
	\node[fill=black]  (18\k) at (1+\k,-2) {};
}
\node[fill=black]  (194) at (6,-3) {};


\foreach \k in {1,...,3}{
	\foreach \j in {1,...,9}{
		\pgfmathtruncatemacro{\next}{1+\j}
		\draw[thick] (1\j\k) to node[draw=none] {} 	(1\next\k);
	}
}
\draw[thick] (103) to node[draw=none] {} 	(113);
\draw[thick] (104) to node[draw=none] {} 	(114);
\foreach \j in {3,...,5}{
	\pgfmathtruncatemacro{\next}{1+\j}
	\draw[thick] (1\j4) to node[draw=none] {} 	(1\next4);
}
\draw[thick] (184) to node[draw=none] {} 	(194);
\foreach \k in {1,...,2}{
	\foreach \j in {1,...,10}{
		\pgfmathtruncatemacro{\nexxt}{1+\k}
		\draw[thick] (1\j\k) to node[draw=none] {} 	(1\j\nexxt);
	}
}

\foreach \k in {3,...,4}{
	\pgfmathtruncatemacro{\nexxt}{1+\k}
	\draw[thick] (10\k) to node[draw=none] {}(10\nexxt);	
	\draw[thick] (13\k) to node[draw=none] {}(13\nexxt);
	\draw[thick] (15\k) to node[draw=none] {}(15\nexxt);
	\draw[thick] (18\k) to node[draw=none] {}(18\nexxt);
}
\draw[thick] (113) to node[draw=none] {}(114);	
\draw[thick] (143) to node[draw=none] {}(144);
\draw[thick] (163) to node[draw=none] {}(164);
\draw[thick] (193) to node[draw=none] {}(194);

\foreach \k in {1,...,2}{
	\foreach \j in {2,...,10}{
		\pgfmathtruncatemacro{\nexxt}{1+\k}
		\pgfmathtruncatemacro{\bef}{\j-1}
		\draw[thick] (1\j\k) to node[draw=none] {} 	(1\bef\nexxt);
	}
}

\draw[thick] (112) to node[draw=none] {} 	(103);
\foreach \k in {3,...,4}{
	\pgfmathtruncatemacro{\nexxt}{1+\k}
	\draw[thick] (11\k) to node[draw=none] {} 	(10\nexxt);
	\draw[thick] (14\k) to node[draw=none] {} 	(13\nexxt);
	\draw[thick] (16\k) to node[draw=none] {} 	(15\nexxt);
	\draw[thick] (19\k) to node[draw=none] {} 	(18\nexxt);
}

\draw[thick] (123) to node[draw=none] {} 	(114);
\draw[thick] (153) to node[draw=none] {} 	(144);
\draw[thick] (173) to node[draw=none] {} 	(164);
\draw[thick] (1103) to node[draw=none] {} 	(194);

\foreach \k in {1,...,1}{
	\foreach \j in {1,...,10}{
		\pgfmathtruncatemacro{\nexxxt}{2+\k}
		\pgfmathtruncatemacro{\bef}{\j-1}
		\draw[thick] (1\j\k) to node[draw=none] {} 	(1\bef\nexxxt);
	}
}

\draw[thick] (112) to node[draw=none] {} 	(104);
\draw[thick] (113) to node[draw=none] {} 	(105);
\draw[thick] (122) to node[draw=none] {} 	(114);	
\draw[thick] (142) to node[draw=none] {} 	(134);
\draw[thick] (143) to node[draw=none] {} 	(135);
\draw[thick] (152) to node[draw=none] {} 	(144);
\draw[thick] (162) to node[draw=none] {} 	(154);
\draw[thick] (163) to node[draw=none] {} 	(155);
\draw[thick] (172) to node[draw=none] {} 	(164);
\draw[thick] (192) to node[draw=none] {} 	(184);
\draw[thick] (193) to node[draw=none] {} 	(185);
\draw[thick] (1102) to node[draw=none] {} 	(194);

\foreach \k in {1, ..., 3}{
	\draw[thick] (1111) to node[draw=none]{}(110\k);
}

\end{tikzpicture}
\caption{The graph $G_2$ shown in black with threshold dimension $2$ and  threshold strong dimension  $4$.}
 \label{fig:OExample}
\end{figure}

\subsubsection{The threshold strong dimension and the strong isometric dimension}
We showed in Corollary \ref{ThresholdStrongEmbedding} that, for a connected graph $G$,  $\tau_s(G)$ is the smallest cardinality of a set $W \subseteq V(G)$ for which there is a $W$-resolved embedding $\varphi(G)$ of $G$ in $P^{\boxtimes,|W|}$, for $P$ a path of sufficiently large order, such that $\varphi(G)$ is an isometric subgraph of $P^{\boxtimes,|W|}$. It is natural to ask if the threshold strong dimension of a graph $G$ has a relationship with the {\em strong isometric dimension} of $G$, denoted by $sdim(G)$, and defined as the smallest integer $n$ such that there is an isometric embedding of $G$ in $P^{\boxtimes,n}$ for some path $P$. Some results pertaining to the strong isometric dimension have been summarized, for example, in \cite{HammackImrichKlavzar2011}. In particular, Theorem 15.4 in \cite{HammackImrichKlavzar2011}, states that $sdim(K_n) =\lceil \log_2(n) \rceil$. However, $\beta_s(K_{n}) = n-1 =\tau_s(K_n)$.  On the other hand Theorem 15.4 in \cite{HammackImrichKlavzar2011} states that $sdim(C_n) = \lceil n/2 \rceil$ whereas $\tau_s(C_n) = 2$ as illustrated by the $\{w_1,w_2\}$-resolved embedding of $C_n$ in $P_{\lceil \frac{n+1}{2} \rceil}^{\boxtimes,2}$ shown in Figure \ref{cycle-embedding}. Thus $sdim(G)$ and $\tau_s(G)$ are distinct parameters.  In fact, there is no general order relationship between $sdim(G)$ and $\tau_s(G)$.

\begin{figure}[h]
\centering
	\begin{minipage}[c]{0.5\textwidth}
		\centering
		$n$ odd	\begin{tikzpicture}[scale=.7,every node/.style={draw,shape=circle,outer sep=2pt,inner sep=1pt,minimum
			size=.2cm}]
		\foreach \x in {0,...,5}
		\foreach \y in {0,...,5}
		{
			\vertex[opacity=0.25]  (\x\y) at (\x,\y) {};
		}
		\foreach \x in {0,...,2}
		\foreach \y in {2,...,4}
		{
			\pgfmathtruncatemacro{\a}{\x+1}
			\pgfmathtruncatemacro{\b}{\y+1}
			\path[opacity=0.25]
			(\x\y) edge (\a\b)
			(\x\b) edge (\a\y);
		}
		\foreach \x in {0,...,2}
		\foreach \y in {2,...,4}
		{
			\pgfmathtruncatemacro{\a}{\x+1}
			\pgfmathtruncatemacro{\b}{\y+1}
			\path[opacity=0.25]
			(\x\y) edge (\a\y)
			(\x\y) edge (\x\b);
		}
		\draw[opacity=0.25]	(32) -- (33)--(34)--(35);
		\draw[opacity=0.25]	(05) -- (15)--(25)--(35);

		\foreach \x in {0,...,2}
		\foreach \y in {0,...,1}
		{
			\pgfmathtruncatemacro{\a}{\x+1}
			\pgfmathtruncatemacro{\b}{\y+1}
			\path[opacity=0.25]
			(\x\y) edge (\a\y);
		}
		\draw[opacity=0.25]	(00) -- (01)--(10)--(11)--(20)--(21)--(30)--(31)--(20);
		\draw[opacity=0.25]	(00) -- (10)--(20)--(30);
		\draw[opacity=0.25]	(01) -- (11)--(21)--(31);
		\draw[opacity=0.25]	(00) -- (11);
		\draw[opacity=0.25]	(10) -- (21);
		
		\foreach \x in {4,...,5}
		\foreach \y in {2,...,4}
		{
			\pgfmathtruncatemacro{\b}{\y+1}
			\path[opacity=0.25]
			(\x\y) edge (\x\b);
		}
		\foreach \x in {4,...,4}
		\foreach \y in {2,...,5}
		{
			\pgfmathtruncatemacro{\a}{\x+1}
			\pgfmathtruncatemacro{\b}{\y+1}
			\path[opacity=0.25]
			(\x\y) edge (\a\y);
		}
		\foreach \x in {4,...,4}
		\foreach \y in {2,...,4}
		{
			\pgfmathtruncatemacro{\a}{\x+1}
			\pgfmathtruncatemacro{\b}{\y+1}
			\path[opacity=0.25](\x\y) edge (\a\b);
		}
		
		\foreach \x in {4,...,4}
		\foreach \y in {3,...,5}
		{
			\pgfmathtruncatemacro{\a}{\x+1}
			\pgfmathtruncatemacro{\b}{\y-1}
			\path[opacity=0.25]
			(\x\y) edge (\a\b);
		}
		
		\foreach \x in {0,...,5}
		\foreach \y in {1,...,1}
		{
			\pgfmathtruncatemacro{\b}{\y+1}
			\draw[ dotted](\x\y) -- (\x\b);
		}
		
		\foreach \x in {3,...,3}
		\foreach \y in {0,...,5}
		{
			\pgfmathtruncatemacro{\a}{\x+1}
			\draw[ dotted](\x\y) -- (\a\y);
		}
		
		\draw[opacity=0.25]	(40) -- (50)--(51)--(40)--(41)--(51);
		\draw[thick, dotted]	(32) -- (41);

		\foreach \j in {0,...,5}{
			\pgfmathtruncatemacro{\a}{5-\j}
			\node[fill=black]  (\j\a) at (\j,\a) {};
		}
		
		
		\node[fill=black]  (15) at (1,5) {};
		\node[fill=black]  (24) at (2,4) {};
		\node[fill=black]  (33) at (3,3) {};
		\node[fill=black]  (51) at (5,1) {};
		
		\draw[thick] (05) --(32);
		\draw[thick] (41)--(50);
		\draw[thick] (15)--(33);

		\draw[thick]  (05) --(15) {};
		\draw[dashed]  (14) --(24) {};
		\draw[dashed]  (23) --(33) {};
		\draw[dashed]  (41) --(51) {};
		
		\draw[dashed]  (14) --(15) {};
		\draw[dashed]  (23) --(24) {};
		\draw[dashed]  (32) --(33) {};
		\draw[thick]  (50) --(51) {};
		
		\draw[thick, dotted] (3,3)--(5,1){};
		
		\node[draw=none] at (-1.8, 5){\small{$w_1 =\left(0, \frac{n-1}{2}\right)$}};
		\node[draw=none] at (5,-0.6){\small{$w_2 =\left( \frac{n-1}{2},0\right)$}};
		\end{tikzpicture}
\end{minipage}%
\begin{minipage}[c]{0.5\textwidth}
\centering
$n$ even	
\begin{tikzpicture}[scale=.7,every node/.style={draw,shape=circle,outer sep=2pt,inner sep=1pt,minimum
	size=.2cm}]
\foreach \x in {0,...,4}
\foreach \y in {0,...,5}
{
	\vertex[opacity=0.25]  (\x\y) at (\x,\y) {};
}
\foreach \x in {0,...,1}
\foreach \y in {2,...,4}
{
	\pgfmathtruncatemacro{\a}{\x+1}
	\pgfmathtruncatemacro{\b}{\y+1}
	\path[opacity=0.25]
	(\x\y) edge (\a\b)
	(\x\b) edge (\a\y);
}
\foreach \x in {0,...,1}
\foreach \y in {2,...,4}
{
	\pgfmathtruncatemacro{\a}{\x+1}
	\pgfmathtruncatemacro{\b}{\y+1}
	\path[opacity=0.25] (\x\y)edge(\x\b)
	(\x\y) edge (\a\y);
}
\draw[opacity=0.25]	(22) -- (23)--(24)--(25);
\draw[opacity=0.25]	(05) -- (15)--(25);

%
\draw[opacity=0.25]	(00) -- (01)--(10)--(11)--(20)--(21);
\draw[opacity=0.25]	(00) -- (10)--(20);
\draw[opacity=0.25]	(01) -- (11)--(21);
\draw[opacity=0.25]	(00) -- (11);
\draw[opacity=0.25]	(10) -- (21);

\foreach \x in {3,...,4}
\foreach \y in {2,...,4}
{
	\pgfmathtruncatemacro{\b}{\y+1}
	\path[opacity=0.25]
	(\x\y) edge (\x\b);
}
\foreach \x in {3,...,3}
\foreach \y in {2,...,5}
{
	\pgfmathtruncatemacro{\a}{\x+1}
	\pgfmathtruncatemacro{\b}{\y+1}
	\path[opacity=0.25]
	(\x\y) edge (\a\y);
}
\foreach \x in {3,...,3}
\foreach \y in {2,...,4}
{
	\pgfmathtruncatemacro{\a}{\x+1}
	\pgfmathtruncatemacro{\b}{\y+1}
	\path[opacity=0.25](\x\y) edge (\a\b);
}

\foreach \x in {3,...,3}
\foreach \y in {3,...,5}
{
	\pgfmathtruncatemacro{\a}{\x+1}
	\pgfmathtruncatemacro{\b}{\y-1}
	\path[opacity=0.25]
	(\x\y) edge (\a\b);
}

\foreach \x in {0,...,4}
\foreach \y in {1,...,1}
{
	\pgfmathtruncatemacro{\b}{\y+1}
	\draw[ dotted](\x\y) -- (\x\b);
}

\foreach \x in {2,...,2}
\foreach \y in {0,...,5}
{
	\pgfmathtruncatemacro{\a}{\x+1}
	\draw[ dotted](\x\y) -- (\a\y);
}

\draw[opacity=0.25]	(30) -- (40)--(41)--(30)--(31)--(41);
\draw[thick, dotted]	(22) -- (31);

\node[fill=black]  (04) at (0,4) {};
\node[fill=black]  (13) at (1,3) {};
\node[fill=black]  (14) at (1,4) {};
\node[fill=black]  (15) at (1,5) {};
\node[fill=black]  (22) at (2,2) {};
\node[fill=black]  (23) at (2,3) {};
\node[fill=black]  (31) at (3,1) {};
\node[fill=black]  (40) at (4,0) {};
\node[fill=black]  (41) at (4,1) {};

\draw[thick] (04) --(22);
\draw[thick] (31)--(40);
\draw[thick] (14)--(23);
\draw[thick] (04)--(15);

\draw[dashed]  (04) --(14) {};
\draw[dashed]  (13) --(23) {};
\draw[dashed]  (31) --(41) {};

\draw[thick]  (14) --(15) {};
\draw[dashed]  (13) --(14) {};
\draw[dashed]  (22) --(23) {};
\draw[thick]  (40) --(41) {};

\draw[thick, dotted] (2,3) --(4,1){};

\node[draw=none] at (-1.8, 4){\small{$w_1 =\left(0, \frac{n-2}{2}\right)$}};
\node[draw=none] at (4,-0.6){\small{$w_2 =\left( \frac{n-2}{2},0\right)$}};

\end{tikzpicture}
\end{minipage}
    \caption{A $\{w_1,w_2\}$-resolved embedding $\varphi(C_n)$ in $P^{\boxtimes,2}_{\lceil \frac{n+1}{2} \rceil}$ that is an isometric subgraph of $P^{\boxtimes,2}_{\lceil \frac{n+1}{2} \rceil}$, where the dashed edges represent the edges we add to $C_n$ to obtain the embedding $\varphi(C_n)$.}
    \label{cycle-embedding}
\end{figure}
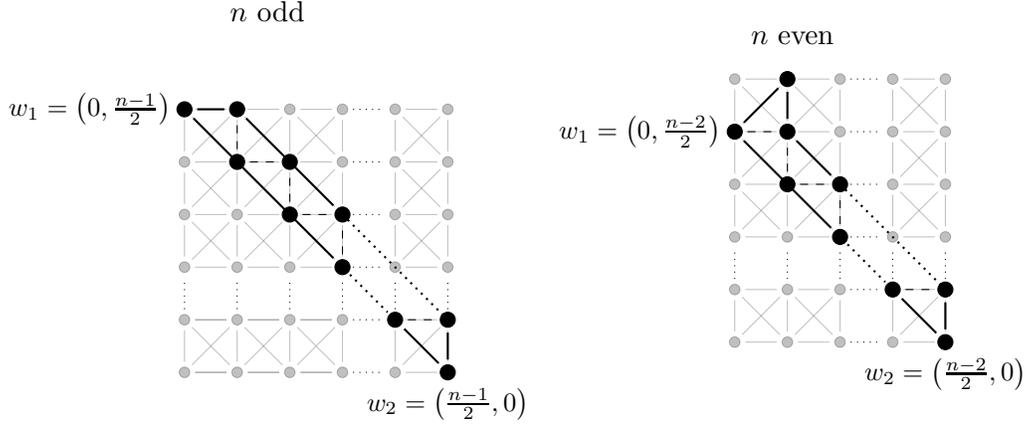

\section{Graphs with vertex covering number 2 that are realizable as strong resolving  graphs}\label{vertexcoveringnumber2}

It is well-known that a graph has strong dimension $1$ if and only if it is a path. Thus paths are $\beta_s$-irreducible graphs of strong dimension $1$ and these are the only such graphs. Thus all graphs with strong dimension $2$ are also $\beta_s$-irreducible. In this section we completely describe the strong resolving graphs for graphs of strong dimension $2$.  In \cite{KuziakPuertas_etal2018} the authors posed the problem of determining which graphs can be realized as strong resolving graphs of some graph. They conjectured that the complete bipartite graphs $K_{s,r}$ where $s,r \ge 2$ are not realizable as strong resolving graphs. Their conjecture was settled in \cite{Lenin2019}. We show next that if a graph has strong dimension $2$, then its strong resolving graph does not contain $K_{2,2}$ as subgraph.

\begin{lemma}
	Let $G$ be a graph such that $\alpha(G_{SR})=2$ and let $w_1$ and $w_2$ be a vertex cover of $G_{SR}$. Then $w_1$ and $w_2$ have at most one common neighbour in $G_{SR}$. 
\end{lemma}
\begin{proof}
	Suppose $w_1$ and $w_2$ have two common neighbours in $G_{SR}$, say vertices $u$ and $v$. Since $\{w_1, w_2\}$ is a strong resolving set for $G$, one of $w_1$ and $w_2$ strongly resolves $u$ and $v$, say $w_1$. So, either $u$ lies on a shortest $w_1-v$ path or $v$ lies on a shortest $w_1-u$ path. We may assume the former. Then $u$ has a neighbour on a shortest $w_1-v$ path that contains $u$, that is further from $w_1$ then $u$. This implies that $u$ is not maximally  distant from $w_1$ and hence $u$ and $w_1$ are not MMD in $G_{SR}$. This is contrary to the assumption that $w_1u \in E(G_{SR})$.
\end{proof}

\begin{remark}
We note that Lemma 4.1 establishes that if a graph $G$ has strong dimension $2$,  then its strong resolving graph does not contain $K_{2,2}$ as a subgraph.
\end{remark}

Let $\{w_1,w_2\}$ be a vertex cover of the strong resolving graph of a graph with strong dimension $2$. From the above lemma, we see that the only possible candidates for graphs with vertex covering number 2 that can be realized by the strong resolving graph of some graphs $G$ fall into one of four categories. We describe these below and in each case construct a graph that has the given graph as its strong resolving graph. In order to describe these constructions, we will use subgraphs of strong products of paths. To this end we assume that the vertices of  a path $P_k$ of order $k$ have been labeled $0,1, \ldots, k-1$ and whenever considering the strong product $P_k \boxtimes P_l$ we will assume that it has been indrawn in the plane so that a vertex $(x,y)$ of this strong product is positioned at the point $(x,y)$ in the plane.

\begin{enumerate}
	\item[Type 1: ] $G_{SR}$ is the disjoint union of two stars, $K_{1,m} \cup K_{1,n}$, where $1\leq m\leq n$. See Figure 4.

	\begin{figure}[tbh]
		\centering
	\begin{tikzpicture}[scale=.5,every node/.style={draw,shape=circle,outer sep=2pt,inner sep=1pt,minimum
		size=.2cm}]
	
	\node(a1) at (0,0){};
	
	\foreach \d in {1,...,4} {
		\node (b\d) at ($(0,0)+(90-\d*22.5:20mm)$) {};
	}
	\node[draw=none] at ($(0,0)+(90-5*22.5:20mm)$){\vdots};
	\node(b5) at ($(0,0)+(90-6*22.5:20mm)$){};
	\foreach \d in {1,...,5} {
		\draw[thick] (a1) --(b\d){};
	}
	
	\node(a2) at (-2,0){};
	
	\foreach \n in {12,..., 15} {
		\node (b\n) at ($(-2,0)+(90-\n*22.5:20mm)$) {};
	}
	\node[draw=none] at ($(-2,0)+(90-11*22.5:20mm)$){\vdots};
	\node(b11) at ($(-2,0)+(90-10*22.5:20mm)$){};
	\foreach \d in {11,...,15} {
		\draw[thick] (a2) --(b\d){};
	}
	\node[draw=none] at (-2.1,-2){$K_{1,m}$};
	\node[draw=none] at (0.9,-2){$K_{1,n}$};
		\node[draw=none] at (-4.5,-0.8){$m$};
	\node[draw=none] at (2.5,-0.8){$n$};
	\end{tikzpicture}
	\caption{Type 1 resolving graph}
\end{figure}
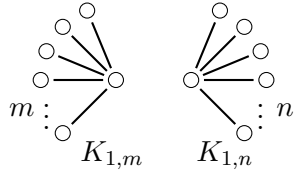


\noindent {$\bullet$} Case 1: $m$ and $n$ have the same parity.  In this case let $H=P_{n+2+\frac{n-m}{2} } \boxtimes P_{n+2}$. Let $G$ be the subgraph of $H$ induced by the vertices on the boundary and in the interior of the region  bounded by the following paths:
 
 \begin{align*}
 Q_1:&(0,n),(1,n-1), \ldots , (n,0)\\
 Q_2:&(0,n), (1,n+1), \ldots, (\frac{n-m}{2}+1, n+\frac{n-m}{2}+1)\\
 Q_3:&(\frac{n-m}{2}+1, n+\frac{n-m}{2}+1), (\frac{n-m}{2}+2, n + \frac{n-m}{2}+1), \ldots, (\frac{n+m}{2}, n+\frac{n-m}{2}+1)\\
 Q_4:&(\frac{n+m}{2}, n+\frac{n-m}{2}+1), (\frac{n+m}{2}+1, n+\frac{n-m}{2}), \ldots, (n+1,n)\\
 Q_5:&(n+1,n),(n+1,n-1), \ldots, (n+1,1)\\
 Q_6:&(n+1,1),(n,0)
 \end{align*}
 
 Then the graph $G$ has the property that $G_{SR} \cong K_{1,m} \cup K_{1,n}$. Figure \ref{strong_resolving_union_of_stars_same_parity} illustrates the construction with $m=n=6$.
 
 \begin{figure}[tbh]
 \begin{center}
 \begin{tikzpicture}[scale=.5,every node/.style={draw,shape=circle,outer sep=2pt,inner sep=1pt,minimum
 	size=.2cm}]
 \foreach \x in {0,...,7}
 \foreach \y in {0,...,7}
 {
 	\vertex[opacity=0.25]  (\x\y) at (\x,\y) {};
 }
 \foreach \x in {0,...,6}
 \foreach \y in {0,...,6}
 {
 	\pgfmathtruncatemacro{\a}{\x+1}
 	\pgfmathtruncatemacro{\b}{\y+1}
 	\path[opacity=0.25]
 	(\x\y) edge (\a\b)
 	(\x\b) edge (\a\y);
 }
 \foreach \x in {0,...,6}
 \foreach \y in {0,...,7}
 {
 	\pgfmathtruncatemacro{\a}{\x+1}
 	\path[opacity=0.25]
 	(\x\y) edge (\a\y)
 	(\y\x) edge (\y\a);
 }
 
 \foreach \j in {1,...,7}{
 	 		\node[fill=black]  (0\j) at (\j-1,7-\j) {};
 	 	}
  	
  \foreach \j in {1,...,6}{
  	\node[fill=black]  (1\j) at (\j,7-\j) {};
  }

 \foreach \j in {1,...,7}{
	\node[fill=black]  (2\j) at (\j,8-\j) {};
}
\foreach \j in {1,...,6}{
	\node[fill=black]  (3\j) at (\j+1,8-\j) {};
}
\foreach \j in {1,...,5}{
	\node[fill=black]  (4\j) at (\j+2,8-\j) {};
}
\foreach \j in {1,...,4}{
	\node[fill=black]  (5\j) at (\j+3,8-\j) {};
}
\foreach \j in {1,...,3}{
	\node[fill=black]  (6\j) at (\j+4,8-\j) {};
}
\foreach \j in {1,...,2}{
	\node[fill=black]  (7\j) at (\j+5,8-\j) {};
}
 \foreach \k in {2,...,7}{
 
 		\pgfmathtruncatemacro{\b}{9-\k}
 		\draw[thick] (\k1) --	(\k\b);
 	}
 
  \foreach \k in {0,...,1}{
 	\pgfmathtruncatemacro{\c}{7-\k}
 	\draw[thick] (\k1) --	(\k\c);
 }

\draw[thick] (21) --	(71);
  \foreach \k in {1,...,6}{
	\pgfmathtruncatemacro{\d}{\k+1}
	\pgfmathtruncatemacro{\e}{8-\k}
	\draw[thick] (0\k) --	(\e\d);
}
\draw[thick] (01) --	(21);
\draw[thick] (02) --	(41);
\draw[thick] (03) --	(61);
\draw[thick] (04) --	(62);
\draw[thick] (05) --	(63);
\draw[thick] (06) --	(45);
\draw[thick] (07) --	(27);

\draw[thick] (11) --	(31);
\draw[thick] (12) --	(51);
\draw[thick] (13) --	(71);
\draw[thick] (14) --	(72);
\draw[thick] (15) --	(54);
\draw[thick] (16) --	(36);

\draw[thick] (27) --	(72);

\foreach \k in {2,...,7}{
	\draw[thick] (0\k) --	(\k1);
}

 \end{tikzpicture}
 \end{center}
 \caption{A graph with strong resolving graph $K_{1,6 } \cup K_{1,6}$}
 \label{strong_resolving_union_of_stars_same_parity}
 \end{figure}
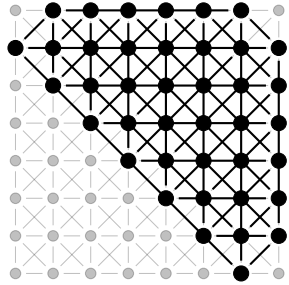
 
\noindent $\bullet$ Case 2: $m$ and $n$ have opposite parity and $1 \le m <n$.  Let $H=P_{n+2+\frac{n-m-1}{2} } \boxtimes P_{n+1}$. We now describe a graph $G$ as an induced subgraph of $H$ using the following paths.

 \begin{align*}
 Q_1:&(0,n),(1,n-1), \ldots , (n,0)\\
 Q_2:&(0,n), (1,n+1), \ldots, (\frac{n-m+1}{2}, \frac{3n-m+1}{2})\\
 Q_3:&(\frac{n-m+1}{2}, \frac{3n-m+1}{2}), (\frac{n-m+3}{2},  \frac{3n-m+1}{2}), \ldots, (\frac{n+m-1}{2}, \frac{3n-m+1}{2})\\
 Q_4:&(\frac{n+m-1}{2}, \frac{ 3n-m+1}{2}), (\frac{n+m-1}{2}+1, \frac{3n-m-1}{2}), \ldots, (n,n)\\
 Q_5:&(n,n),(n,n-1), \ldots, (n,0)
 \end{align*}
 
 Let $G$ be the subgraph of $H$ induced by the vertices on the boundary and in the interior of the region bounded by these paths $Q_1$, ..., $Q_5$. Then $G_{SR} \cong K_{1,m} \cup K_{1,n}$. Figure  \ref{strong_resolving_union_of_stars_opposite_parity} illustrates the construction for $m=5$ and $n=6$.

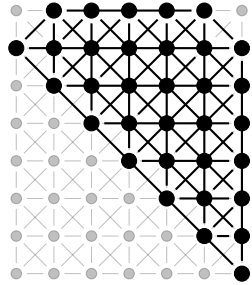
\begin{figure}[tbh]
 \begin{center}
  	 \begin{tikzpicture}[scale=.5,every node/.style={draw,shape=circle,outer sep=2pt,inner sep=1pt,minimum
 		size=.2cm}]
 	\foreach \x in {0,...,6}
 	\foreach \y in {0,...,6}
 	{
 		\vertex[opacity=0.25]  (\x\y) at (\x,\y) {};
 	}
 	
 	\foreach \x in {0,...,6}
 	{
 		\vertex[opacity=0.25]  (\x7) at (\x,7) {};
 	}
 	\draw[opacity=0.25]  (0,6.3) -- (0,6.7);
 	\draw[opacity=0.25]  (0.3,7) -- (0.7,7);
 	\draw[opacity=0.25]  (0.3,6.7) -- (0.7,6.3);
 	\draw[opacity=0.25]  (6,6.3) -- (6,6.7);
 	\draw[opacity=0.25]  (5.3,7) -- (5.7,7);
 	\draw[opacity=0.25]  (5.3,6.3) -- (5.7,6.7);

 	\foreach \x in {0,...,5}
 	\foreach \y in {0,...,5}
 	{
 		\pgfmathtruncatemacro{\a}{\x+1}
 		\pgfmathtruncatemacro{\b}{\y+1}
 		\path[opacity=0.25]
 		(\x\y) edge (\a\b)
 		(\x\b) edge (\a\y);
 	}
 	\foreach \x in {0,...,5}
 	\foreach \y in {0,...,6}
 	{
 		\pgfmathtruncatemacro{\a}{\x+1}
 		\path[opacity=0.25]
 		(\x\y) edge (\a\y)
 		(\y\x) edge (\y\a);
 	}
 	
 	\foreach \j in {1,...,7}{
 		\node[fill=black]  (0\j) at (\j-1,7-\j) {};
 	}
 	
 	\foreach \j in {1,...,6}{
 		\node[fill=black]  (1\j) at (\j,7-\j) {};
 	}
 	
 	\foreach \j in {1,...,6}{
 		\node[fill=black]  (2\j) at (\j,8-\j) {};
 	}
 	\foreach \j in {1,...,5}{
 		\node[fill=black]  (3\j) at (\j+1,8-\j) {};
 	}
 	\foreach \j in {1,...,4}{
 		\node[fill=black]  (4\j) at (\j+2,8-\j) {};
 	}
 	\foreach \j in {1,...,3}{
 		\node[fill=black]  (5\j) at (\j+3,8-\j) {};
 	}
 	\foreach \j in {1,...,2}{
 		\node[fill=black]  (6\j) at (\j+4,8-\j) {};
 	}
 	\foreach \k in {2,...,6}{
 		
 		\pgfmathtruncatemacro{\b}{8-\k}
 		\draw[thick] (\k1) --	(\k\b);
 	}
 	
 	\foreach \k in {0,...,1}{
 		\pgfmathtruncatemacro{\c}{7-\k}
 		\draw[thick] (\k1) --	(\k\c);
 	}
 	%
 	\draw[thick] (21) --	(61);
 	\draw[thick] (06) --	(16);
 	\foreach \k in {1,...,5}{
 		\pgfmathtruncatemacro{\d}{\k+1}
 		\pgfmathtruncatemacro{\e}{7-\k}
 		\draw[thick] (0\k) --	(\e\d);
 	}
 	\draw[thick] (01) --	(21);
 	\draw[thick] (02) --	(41);
 	\draw[thick] (03) --	(61);
 	\draw[thick] (04) --	(62);
 	\draw[thick] (05) --	(44);
 	\draw[thick] (06) --	(26);
 	%
 	\draw[thick] (11) --	(31);
 	\draw[thick] (12) --	(51);
 	\draw[thick] (13) --	(52);
 	\draw[thick] (14) --	(53);
 	\draw[thick] (15) --	(35);

 	\draw[thick] (07) --	(62);
 	
 	\foreach \k in {2,...,6}{
 		\draw[thick] (0\k) --	(\k1);
 	} 
 	\end{tikzpicture}
 \caption{A graph with strong resolving graph $K_{1,5 } \cup K_{1,6}$}
 \label{strong_resolving_union_of_stars_opposite_parity}
 \end{center}

	\end{figure}

	\item[Type 2:] $G_{SR}$ is $K_{1,m} \cup K_{1,n}$ + $w_1w_2$ for ($1\leq m\leq n$), where $w_1$ and $w_2$ are the centers of the  stars $K_{1,m}$ and $K_{1,n}$. See Figure 7.
	
	\begin{figure}[h]
	\begin{center}
	\begin{tikzpicture}[scale=.5,every node/.style={draw,shape=circle,outer sep=2pt,inner sep=1pt,minimum
			size=.2cm}]
		
		\node(a1) at (0,0){};
		
		\foreach \d in {1,...,4} {
			\node (b\d) at ($(0,0)+(90-\d*22.5:20mm)$) {};
		}
		\node[draw=none] at ($(0,0)+(90-5*22.5:20mm)$){\vdots};
		\node[draw=none] at (0,-0.7){$w_2$};
		\node(b5) at ($(0,0)+(90-6*22.5:20mm)$){};
		\foreach \d in {1,...,5} {
			\draw[thick] (a1) --(b\d){};
		}
		
		\node(a2) at (-2,0){};
		
		\foreach \n in {12,..., 15} {
			\node (b\n) at ($(-2,0)+(90-\n*22.5:20mm)$) {};
		}
		\node[draw=none] at ($(-2,0)+(90-11*22.5:20mm)$){\vdots};
		\node[draw=none] at (-2,-0.7){$w_1$};
		\node(b11) at ($(-2,0)+(90-10*22.5:20mm)$){};
		\foreach \d in {11,...,15} {
			\draw[thick] (a2) --(b\d){};
		}
		\draw[thick] (a1) --(a2){};
		\node[draw=none] at (-4.5,-0.8){$m$};
		\node[draw=none] at (2.5,-0.8){$n$};
		\end{tikzpicture}
		\caption{Type 2 strong resolving graph}
		\end{center}
		\end{figure}
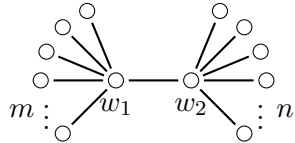

	If $m$ and $n$ have the same parity, then let $G$ be obtained from the graph described in Case 1 of Type 1 by adding a leaf adjacent to each of the vertices  $(0,n)$ and $(n,0)$. These two new leaves then become $w_1$ and $w_2$.  If $m$ and $n$ have opposite parity, then let $G$ be obtained from the graph described in Case 2 of Type 1 by adding a leaf adjacent to $(0,n)$. Then $G_{SR}  \cong  (K_{1,m} \cup K_{1,n} + w_1w_2)$ where $w_1$ and $w_2$ are the centers of the two stars in the union.

	\item[Type 3:] $G_{SR}$ is $K_{1,m} \cup K_{1,n} + \{vw_1,vw_2\}$  for ($1\leq m\leq n$), where $w_1$ and $w_2$ are the centers of the two stars  and $v$ is a new vertex.  See Figure 8.
	
		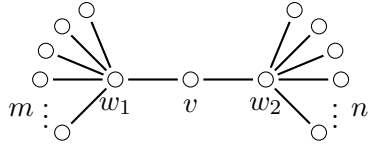
\begin{figure}[h]
		\begin{center}
		\begin{tikzpicture}[scale=.5,every node/.style={draw,shape=circle,outer sep=2pt,inner sep=1pt,minimum
			size=.2cm}]
		
		\node(a1) at (2,0){};
		\node(v) at(0,0){};
		\node[draw=none] at (0,-0.7){$v$};
		
		\foreach \d in {1,...,4} {
			\node (b\d) at ($(2,0)+(90-\d*22.5:20mm)$) {};
		}
		\node[draw=none] at ($(2,0)+(90-5*22.5:20mm)$){\vdots};
		\node[draw=none] at (2,-0.7){$w_2$};
		\node(b5) at ($(2,0)+(90-6*22.5:20mm)$){};
		\foreach \d in {1,...,5} {
			\draw[thick] (a1) --(b\d){};
		}
		
		\node(a2) at (-2,0){};
		
		\foreach \n in {12,..., 15} {
			\node (b\n) at ($(-2,0)+(90-\n*22.5:20mm)$) {};
		}
		\node[draw=none] at ($(-2,0)+(90-11*22.5:20mm)$){\vdots};
		\node[draw=none] at (-2,-0.7){$w_1$};
		\node(b11) at ($(-2,0)+(90-10*22.5:20mm)$){};
		\foreach \d in {11,...,15} {
			\draw[thick] (a2) --(b\d){};
		}
		\draw[thick] (a1) --(v) -- (a2){};
		\node[draw=none] at (-4.5,-0.8){$m$};
		\node[draw=none] at (4.5,-0.8){$n$};
		\end{tikzpicture}
		\caption{Type 3 strong resolving graph}
	\end{center}
	\end{figure}
	
	If $m$ and $n$ have the same parity we take the graph described in Case 1 for Type 1 and add the vertex $(\frac{n+m+2}{2}, \frac{3n-m+2}{2})$ and join it to each of $(\frac{n+m}{2}, \frac{3n-m+2}{2})$, $(\frac{n+m+2}{2}, \frac{3n-m}{2})$ and $(\frac{n+m}{2},  \frac{3n-m}{2})$. If $m$ and $n$ have opposite parity, then we take the graph constructed in Case 2 for Type 1 and add the vertex $(\frac{n+m+1}{2},  \frac{3n-m+1}{2})$ and join it to 
$(\frac{n+m-1}{2},  \frac{3n-m+1}{2})$, $(\frac{n+m+1}{2},  \frac{3n-m-1}{2})$ and\\ $(\frac{n+m-1}{2},  \frac{3n-m-1}{2})$. In either case let $G$ be the resulting graph. Then $G_{SR} \cong (K_{1,m} \cup K_{1,n}+ vw_1,vw_2)$  for ($1\leq m\leq n$), where $w_1$ and $w_2$ are the centers of the two stars $K_{1,m}$ and $K_{1,n}$ and $v$ is a new vertex. 

\item[Type 4:] $G_{SR}$ is $K_{1,m} \cup K_{1,n} + \{vw_1,vw_2,w_1w_2\}$ for  ($1\leq m\leq n$), where $w_1$ and $w_2$ are the centers of the two stars  and $v$ is a new vertex. See Figure 9.
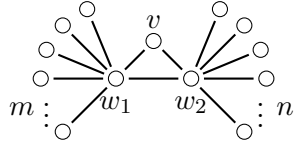
\begin{figure}[h]
	\begin{center}
		\begin{tikzpicture}[scale=.5,every node/.style={draw,shape=circle,outer sep=2pt,inner sep=1pt,minimum
			size=.2cm}]
		
		\node(a1) at (0,0){};
		
		\foreach \d in {1,...,4} {
			\node (b\d) at ($(0,0)+(90-\d*22.5:20mm)$) {};
		}
		\node[draw=none] at ($(0,0)+(90-5*22.5:20mm)$){\vdots};
		\node[draw=none] at (0,-0.7){$w_2$};
		\node(b5) at ($(0,0)+(90-6*22.5:20mm)$){};
		\foreach \d in {1,...,5} {
			\draw[thick] (a1) --(b\d){};
		}
		
		\node(a2) at (-2,0){};
		
		\foreach \n in {12,..., 15} {
			\node (b\n) at ($(-2,0)+(90-\n*22.5:20mm)$) {};
		}
		\node[draw=none] at ($(-2,0)+(90-11*22.5:20mm)$){\vdots};
		\node[draw=none] at (-2,-0.7){$w_1$};
		\node(b11) at ($(-2,0)+(90-10*22.5:20mm)$){};
		\foreach \d in {11,...,15} {
			\draw[thick] (a2) --(b\d){};
		}
		\node[draw=none] at (-4.5,-0.8){$m$};
		\node[draw=none] at (2.5,-0.8){$n$};
		\node(c) at (-1.0,1.0){}; 
		\node[draw=none] at (-1.0, 1.5){$v$}; 
		\draw[thick] (a1) --(a2)--(c) --(a1) {};
		\end{tikzpicture}
		\caption{Type 4 strong resolving graph}
	\end{center}
\end{figure}

In this case we take the graph described in Case 1 or Case 2 of Type 3 above and add a leaf to the vertices $(0,n)$ and $(n,0)$. In either case let $G$ be the resulting graph. Then $G_{SR} \cong (K_{1,m} \cup K_{1,n}+ vw_1,vw_2,w_1w_2)$  for ($1\leq m\leq n$), where $w_1$ and $w_2$ are the centers of the two stars  and $v$ is a new vertex. 
	\end{enumerate}

\section{Bounds} \label{bounds}
It appears to be a difficult problem to determine the threshold strong dimension of a graph. In this section we establish bounds for this invariant for graphs in general and for trees.  Let $\mathcal P(S)$ denote the power set of a set $S$. The complete graph with vertex set $V$ is denoted by $K_V$. If $|V|=n$ then $K_V$ is isomorphic to the complete graph on $n$ vertices, which is denoted by $K_n$. The complete $\ell$-partite graph with partite sets $V_1,V_2,\ldots,V_\ell$ is denoted by $K_{V_1,V_2,\ldots,V_\ell}$. If $G$ is a complete graph of order $n$, then $\tau_s(G)=\beta_s(G)=n-1$. The next result gives an upper bound for $\tau_s(G)$ when $G$ is not complete.

\begin{theorem} Let $G$ be a non-complete graph with $\chi(G) = k$ and let $V_1,V_2,\ldots,V_k$ be the color classes in a proper $k$-coloring of $G$, where $|V_1|\leq |V_2|\leq \cdots \leq |V_k|$. \begin{enumerate}\item If there is an $\ell \geq 1$ such that $|V_i| = 1$ for $1\leq i\leq \ell$ and $|V_i|>1$ for $\ell < i \leq k$, then $\tau_s(G)\leq \ell-1+ \sum_{i=\ell+1}^k\lceil\log_2|V_i|\rceil$.
\item If $|V_1|\geq 2$, then $\tau_s(G)\leq \sum_{i=1}^k\lceil\log_2|V_i|\rceil$.
\end{enumerate}
\end{theorem}

\proof
1. We add edges to $G$ to form a supergraph $H$ for which $\beta_s(H)\leq \ell-1+ \sum_{i=\ell+1}^k\lceil\log_2|V_i|\rceil$. We start by adding all additional edges between the distinct pairs of color classes necessary to form the complete $k$-partite supergraph $K_{V_1,V_2,\ldots,V_k}$ of $G$. This graph has diameter 2, so for vertices $u$ and $v$, $u$MMD$v$ if and only if $u$ and $v$ are universal vertices  or $u$ and $v$ are nonadjacent. Each color class $V_i$ is an independent set. For each $i\in\{\ell+1, \ell+2,\ldots,k\}$ choose a subset $W_i\subseteq V_i$ such that $|W_i| = \lceil\log_2|V_i|\rceil$. Then $|\mathcal P(W_i)| \geq |V_i|$. For each $i\in \{\ell+1,\ldots,k\}$, assign to each $v\in V_i-W_i$ a member of $\mathcal P(W_i)-\{W_i\}$ such that every two vertices in $V_i-W_i$ are assigned distinct subsets of $W_i$. This is possible since $|\mathcal P(W_i) - \{W_i\}| \geq |V_i|-1 \geq |V_i-W_i|$. Add edges to $K_{V_1,V_2,\ldots,V_k}$ so that the set assigned to each vertex $v$ of $V_i-W_i$ is the $W_i$-neighbourhood of $v$ in the resulting graph. So every pair of vertices in $V_i-W_i$ have distinct $W_i$-neighbourhoods. Finally, $H$ is obtained by adding edges between vertices of $V_i-W_i$ so that they form a clique in $H$.

Now we construct the strong resolving graph $H_{SR}$ of $H$. The $\ell$  vertices of $V_1\cup V_2\cup \cdots \cup V_\ell$ are all universal vertices and hence are pairwise MMD. Also, for each $i\in\{\ell+1,\ldots,k\}$, the vertices of $W_i$ are pairwise MMD, and each vertex $v\in V_i-W_i$ is MMD with any vertex in $W_i-N_{W_i}(v)$. There are no other MMD pairs of vertices in $H$. Hence, the strong resolving graph $H_{SR}$ of $H$ contains the clique $K_{V_1\cup V_2\cup \cdots \cup V_\ell}$ $(\cong K_{\ell})$ and the clique $K_{W_i}$ for each $i$ where $\ell+1\leq i\leq k$. Apart from the edges in these cliques, $H_{SR}$ may contain some edges joining vertices in $W_i$ with vertices in $V_i-W_i$, for $\ell+1\leq i\leq k$.  Let $W= (V_1\cup V_2\cup\cdots \cup V_{\ell-1})\cup(W_{\ell+1}\cup W_{\ell+2}\cup \ldots\cup W_k)$. Then $W$ is a minimum vertex covering of $H_{SR}$, and so $\beta_s(H) = |W| = \ell-1+ \sum_{i=\ell+1}^k\lceil\log_2|V_i|\rceil$. Since $G$ is a spanning subgraph of $H$, it follows that $\tau_s(G) \leq \ell-1+ \sum_{i=\ell+1}^k\lceil\log_2|V_i|\rceil$.\\

\noindent 2. The proof of this case is similar to the proof of part 1. We construct a supergraph $H$ of $G$ using the same process as in the proof of case 1, assigning a unique $W_i$-neighbourhood to each vertex of $V_i-W_i$, where $W_i\subseteq V_i$ with $|W_i|=\lceil \log_2|V_i|\rceil$. Then $W_i$ induces a clique in the strong resolving graph $H_{SR}$, for $1\leq i\leq k$. In addition, $H_{SR}$ may contain edges joining vertices of $W_i$ with vertices in $V_i-W_i$. Hence $W = W_1\cup W_2 \cup \cdots\cup W_k$ is a minimum vertex covering of $H_{SR}$, and so $\beta_s(H) = |W| =  \sum_{i=1}^k\lceil\log_2|V_i|\rceil$. Since $G$ is a spanning subgraph of $H$, it follows that $\tau_s(G) \leq \sum_{i=1}^k\lceil\log_2|V_i|\rceil$.
\qed

\bigskip

We now present an improved bound for trees.

\begin{theorem} If $T$ is a tree with $n\geq 2$ vertices, then $\tau_s(T)\leq \lceil \log_2n\rceil$.
\end{theorem}

\proof Suppose a tree $T = (V,E)$ has $n$ vertices and $\ell$ leaves. If $\ell <\lceil\log_2 n\rceil$, then $T_{SR}\cong K_{\ell}$ which has vertex covering number $\ell-1$. In this case  $\beta_s(T)=\ell-1 < \lceil\log_2 n\rceil$. It follows that $\tau_s(T) \leq \lceil \log_2n\rceil$, as claimed. Suppose that $T$ has at least $\lceil \log_2n\rceil$ leaves. We add edges to $T$ to construct a supergraph $H$ for which $\beta_s(H) = \lceil \log_2n\rceil$. Let $W$ be a set of $\lceil\log_2 n\rceil$ leaves of $T$. Then $|\mathcal P(W)|\geq n$. Now assign to each vertex $v\in V-W$ a subset of vertices from $W$ that will be its $W$-neighbourhood in $H$. We do this in such a way that every two vertices of $V-W$ are assigned distinct $W$-neighbourhoods in $H$. Let $V'$ denote the set of vertices in $V-W$ whose $W$-neighbourhoods in $T$ are nonempty. If the $W$-neighbourhood of $v$ in $T$ is nonempty, then $v$ is adjacent to a set of leaves of $T$ that belong to $W$ and no other vertex of $V-W$ has the same leaf neighbours in $T$. In this case the assigned $W$-neighbourhood of $v$ in $H$ remains same as its $W$-neighbourhood in $T$. We now remove these assigned neighbourhoods from $\mathcal P(W)$. Let $\mathcal P'$ be the subset of $\mathcal P(W)$ that remains. To each of the remaining vertices in $(V-W)-V'$ we assign a unique member of $\mathcal P'$ in such a way that no two vertices of $(V-W)-V'$ are assigned the same element of $\mathcal P'$. For each $v\in V-W$, the vertices of the assigned member of $\mathcal P(W)$ becomes its $W$-neighbourhood in $H$. We ensure that exactly one vertex of $V-W$, say $u$, is assigned the whole set $W$ as its $W$-neighbourhood. Observe that this is possible since $|\mathcal P(W)|\geq n >|V-W|$. Finally, we obtain $H$ by adding edges between the vertices of $V-W$ so that they form a clique. Since the resulting graph $H$ has the universal vertex $u\in V-W$, $H$ has diameter 2, and so the MMD pairs of vertices of $H$ are the nonadjacent pairs of vertices. Since $W$ is an independent set, the vertices of $W$ are pairwise MMD, and each vertex $v\in V-W$ is MMD with each vertex in $V-N_W(v)$. There are no other pairs of MMD vertices in $H$. The vertices of $W$ form a clique in the strong resolving graph $H_{SR}$ of $H$. All other edges of $H_{SR}$ join vertices of $W$ with vertices in $V-W$. Thus $W$ is a minimum vertex covering of $H_{SR}$, and so $\beta_s(H) = |W| = \lceil \log_2n\rceil$. Since $T$ is a spanning subgraph of $H$, it follows that $\tau_s(T)\leq \lceil \log_2n\rceil$.\qed

\section{The threshold strong dimension of trees} \label{specialclasses}
In this section we show that for trees with strong dimension $3$ or $4$, the threshold strong dimension is $2$. To see that this does not extend to trees of higher strong dimension, we note that it was observed in \cite{MolMurphyOellermann2019} that $K_{1,6}$ does not have threshold dimension $2$. Since $\beta_s(K_{1,6})=5$ and $\tau(K_{1,6}) \le \tau_s(K_{1,6})$, we see that the threshold strong dimension of trees with strong dimension $5$ need not be $2$.  We also observe that there are trees of arbitrarily large dimension that have threshold strong dimension 2. We use the following known results from \cite{SeboTannier2004}.

\begin{theorem}{\em \cite{SeboTannier2004} }\label{StrongdimTrees}
Let $T$ be a tree. Then  $\beta_s(T) = \ell -1$  if and only if $T$ has $\ell$ leaves.
\end{theorem}

\noindent To prove that the threshold strong dimension, for a tree $T$ with strong dimension 3 or 4,  is $2$, we describe a $W$-resolved embedding of $T$ in $P^{\boxtimes,2}_{D+1}$ where $W$ is a set of two vertices of $T$ and $D$ is the diameter of $T$. By observing that these embeddings are isometric subgraphs of $P^{\boxtimes,2}_{D+1}$  the result follows from Theorem \ref{strong_resolving_characterization}. In order to describe these embeddings, we define two useful notions. We assume, as before that the vertices of the path $P_{D+1}$ have been labeled  $0,1, \ldots, D$ such that $i$ and $i+1$ are adjacent for $0 \le i <D$.

\begin{definition}~\\
\begin{enumerate}
\item $\bullet$ A {\em northwest-southeast} (abbreviated  NW-SE) {\em diagonal} of $P^{\boxtimes,2}_{D+1}$ is the subpath induced by $\{(x,y): x+y=k\}$ for some integer $0 < k \le D$, i.e., it is the unique $(0,k)$--$(k,0)$ geodesic  for some integer $0 < k \le D$, and \\
$\bullet$ a {\em southwest-northeast} (abbreviated SW-NE) {\em diagonal} of $P^{\boxtimes,2}_{D+1}$ is the subpath induced by $\{(x,y): x-y = k\}$ or $\{(x,y): x-y = -k\}$ for some integer $0 \le k < D$, i.e., it is the unique geodesic passing either through $(k,0)$ and $(D, D-k)$ or through $(0,k)$ and $(D-k, D)$, respectively for some integer $0 \le k < D$.
\item $\bullet$ If $(x,y)$ is a vertex on a NW-SE diagonal $Q$,  then the vertices $(a,b)$ of $Q$ satisfying $a<x$ will be referred to as the vertices of $Q$ that lie NW of $(x,y)$ while those vertices $(a,b)$ on $Q$ satisfying $a>x$ will be referred to as the vertices of $Q$ that lie SE of $(x,y)$.\\
$\bullet$ If $(x,y)$ is a vertex on a SW-NE diagonal $Q$, then the vertices $(a,b)$ of $Q$ satisfying $b<y$ will be referred to as the vertices of $Q$ that lie SW of $(x,y)$ while those vertices $(a,b)$ of $Q$ satisfying $b>y$ will be referred to as the vertices of $Q$ that lie NE of $(x,y)$.

\end{enumerate}

\end{definition}

\noindent Note that for every vertex in $P^{\boxtimes,2}_{D+1}$, there is a unique NW-SE diagonal and a unique SW-NE diagonal that passes through it. In the sequel, when illustrating an embedding $\varphi(T)$ of a given tree $T$  in  $P^{\boxtimes,2}_{D+1}$, the solid black edges correspond to the edges of $T$ while the dashed black edges are the edges that are added to $T$ to obtain the embedding.  When illustrating an embedding of  a tree $T$ of diameter $D$, as described in the proofs of our theorems, we may occasionally embed  $T$ in $P^{\boxtimes,2}_{m}$ where $m < D+1$ provided $P^{\boxtimes,2}_{m}$ admits the described embedding.

	\begin{theorem}  \label{strongdim3tree_has_threshold2}
	If $T$ is a tree with $\beta_s(T)=3$, then $\tau_s(T)=2$.
	\end{theorem}

\begin{proof} Let $T$ be a tree with strong dimension 3.  Then, by Theorem \ref{StrongdimTrees}, $T$ has exactly four leaves as shown in Figure \ref{strongdim3} where possibly $k_1=1$.  Let $V(T)=\{v_1,v_2,\ldots,v_{k_1},$ $u_1,u_2,\ldots,u_{k_2},$ $x_1,x_2,\ldots,x_{k_3},$ $y_1,y_2,\ldots,y_{k_4},$ $z_1,z_2,\ldots,z_{k_5}\}$ where $k_i \geq 1$ for $i=1,2,3,4,5$, and let $E(T)=\{v_iv_{i+1}:1\leq i\leq k_1-1\}\cup \{u_iu_{i+1}:1 \leq i \leq k_2-1\}\cup \{x_ix_{i+1}:1 \leq i \leq k_3-1\}\cup \{y_iy_{i+1}:1 \leq i \leq k_4-1\}\cup \{z_iz_{i+1}:1 \leq i \leq k_5-1\} \cup \{v_1y_1,v_1z_1,v_{k_1}u_1,v_{k_1}x_1\}$, see Figure   \ref{strongdim3}.

We may assume without loss of generality that $k_2 \geq k_3$ and $k_4 \geq k_5$.  Let $k = \lceil \frac{k_1+1}{2} \rceil + k_2+k_4-1$. Let $D$ be the diameter of $T$. To describe an embedding of $T$ in $P^{\boxtimes,2}_{D+1}$  assume the vertices of $P_{D+1}$ have been labeled $0,1, \ldots, D$. We consider two cases based on the parity of $k_1$.  

\noindent{\bf Case 1:}  $k_1$ is odd. In this case we place the vertices $y_{k_4}, y_{k_4-1}, \ldots, y_1$, $v_1, v_3, \ldots, v_{k_1-2},$ $ v_{k_1}$, $u_1, u_2, \ldots, u_{k_2}$ (or $y_{k_4}, y_{k_4-1}, \ldots, y_1$, $v_1$, $u_1, u_2, \ldots, u_{k_2}$  if $k_1=1$) in this order along the NW-SE diagonal through $(0,k)$ and $(k,0)$ starting at $(0,k)$. The remaining vertices are placed along the NW-SE diagonal through $(0,k+1)$ and $(k+1,0)$  by  placing $z_{k_5}$  in position $(k_4-k_5+1, k-(k_4-k_5))$ followed by the remaining vertices $z_{k_5-1}, \ldots z_1$, $v_2, v_4, \ldots, v_{k_1-1}, x_1, x_2, \ldots, x_{k_3}$ in that order along this diagonal SE of $z_{k_5}$ (ending with $x_{k_3}$ in position $(k_4+k_3+\frac{k_1-1}{2}, k+1-(k_4+k_3+\frac{k_1-1}{2}))$). This embedding is illustrated in Figure \ref{Case1}.

\noindent{\bf Case 2:}  $k_1$ is even. In this case we place $y_{k_4}, y_{k_4-1}, \ldots, y_1$, $v_1, v_3, \ldots v_{k_1-1}$, $v_{k_1}$, $u_1, u_2, \ldots, u_{k_2}$ if $k_1 \ge 4$ (or $y_{k_4}, y_{k_4-1}, \ldots, y_1$, $v_1, v_2$, $u_1, u_2, \ldots, u_{k_2}$ if $k_1=2$) in this order along the NW-SE diagonal through $(0,k)$ and $(k,0)$ starting at $(0,k)$. The remaining vertices are again placed along the NW-SE diagonal through $(0,k+1)$ and $(k+1,0)$  by first placing $z_{k_5}$  in position $(k_4-k_5+1, k-(k_4-k_5))$ followed by the remaining vertices $z_{k_5-1}, \ldots z_1$, $v_2, v_4, \ldots, v_{k_1-2}, x_1, x_2, \ldots, x_{k_3}$ if $k_1 \ge 4$  (or $z_{k_5-1}, \ldots z_1$, $ x_1, x_2, \ldots, x_{k_3}$ if $k_1=2$), in that order along this diagonal SE of $z_{k_5}$ (ending with $x_{k_3}$ in position $(k_4+k_3+\frac{k_1-2}{2}, k+1-(k_4+k_3+\frac{k_1-2}{2})$). This embedding is illustrated in Figure \ref{Case2}.

\medskip

\noindent In either case the subgraph of $P^{\boxtimes,2}_{D+1}$  induced by the vertices of this embedding is a $\{y_{k_4}, u_{k_2}\}$-resolved embedding of $T$ in $P^{\boxtimes,2}_{D+1}$  that is also an isometric subgraph of $P^{\boxtimes,2}_{D+1}$ . Thus $\tau_s(T) =2$.
\end{proof}

\medskip

\begin{figure}[h]
\centering
\begin{tikzpicture}[scale=.7,every node/.style={draw,shape=circle,outer sep=2pt,inner sep=1pt,minimum
	size=.2cm}]

\foreach \j in {1, ..., 2}{
	\node[fill=black]  (\j1) at (-1-\j,1) {};
	\node[draw=none] at (-1-\j, 1.5){\small{$z_{\j}$}};
}
\foreach \j in {1,...,2}{
	\node[fill=black]  (\j0) at (\j-2,0) {};
	\node[draw=none] at (\j-2, -0.5){\small{$v_{\j}$}};
}

\foreach \j in {1, ..., 2}{
	\node[fill=black]  (\j2) at (-1-\j,-1) {};
	\node[draw=none] at (-1-\j, -1.5){\small{$y_{\j}$}};
}
\foreach \j in {1,...,2}{
	\node[fill=black]  (\j3) at (\j+2,1) {};
	\node[draw=none] at (\j+2, 1.5){\small{$x_{\j}$}};
}
\foreach \j in {1,...,2}{
	\node[fill=black]  (\j4) at (\j+2,-1) {};
	\node[draw=none] at (\j+2, -1.5){\small{$u_{\j}$}};
}
\node[fill=black]  (40) at (2,0) {};
\node[draw=none] at (2, -0.5){\small{$v_{k_1}$}};
\node[fill=black]  (41) at (-5,1) {};
\node[draw=none] at (-5, 1.5){\small{$z_{k_5}$}};
\node[fill=black]  (42) at (-5,-1) {};
\node[draw=none] at (-5, -1.5){\small{$y_{k_5}$}};
\node[fill=black]  (62) at (-7,-1) {};
\node[draw=none] at (-7, -1.5){\small{$y_{k_4}$}};
\node[fill=black]  (43) at (6,1) {};
\node[draw=none] at (6, 1.5){\small{$x_{k_3}$}};
\node[fill=black]  (44) at (6,-1) {};
\node[draw=none] at (6, -1.5){\small{$u_{k_3}$}};
\node[fill=black]  (64) at (8,-1) {};
\node[draw=none] at (8, -1.5){\small{$u_{k_2}$}};

\draw[thick] (21)--(11)--(10)--(20);
\draw[thick] (40)--(13)--(23);
\draw[thick] (22)--(12)--(10);
\draw[thick] (40)--(14)--(24);
\draw[thick] (20) -- (0.7, 0);
\draw[thick, dotted] (0.7,0)--(1.3,0);
\draw[thick] (1.3,0)--(40);
\draw[thick] (41) -- (-4.3, 1);
\draw[thick, dotted] (-4.3,1)--(-3.7,1);
\draw[thick] (-3.7,1)--(21);
\draw[thick] (42) -- (-4.3, -1);
\draw[thick, dotted] (-4.3,-1)--(-3.7,-1);
\draw[thick] (-3.7,-1)--(22);

\draw[thick] (62) -- (-6.3, -1);
\draw[thick, dotted] (-6.3,-1)--(-5.7,-1);
\draw[thick] (-5.7,-1)--(42);

\draw[thick] (23) -- (4.7, 1);
\draw[thick, dotted] (4.7,1)--(5.4,1);
\draw[thick] (5.4,1)--(43);

\draw[thick] (24) -- (4.7, -1);
\draw[thick, dotted] (4.7,-1)--(5.4,-1);
\draw[thick] (5.4,-1)--(44);

\draw[thick] (44) -- (6.7, -1);
\draw[thick, dotted] (6.7,-1)--(7.4,-1);
\draw[thick] (7.4,-1)--(64);

\end{tikzpicture}

\caption{A tree with strong dimension 3} \label{strongdim3}

\end{figure}
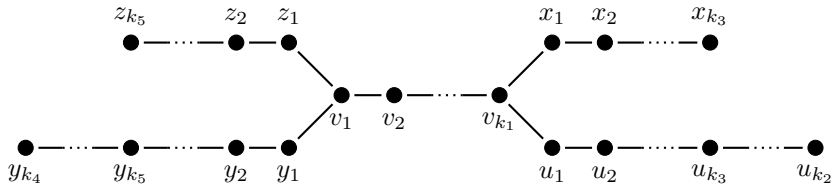

\medskip

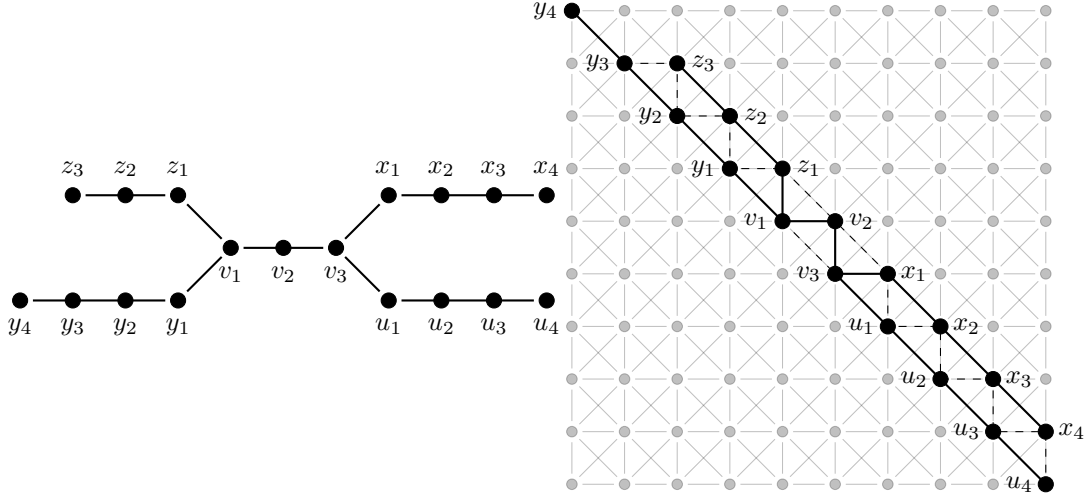
\begin{figure} [h] \label{Case1}
	\begin{minipage}[c]{0.5\textwidth}
		\centering
		\begin{tikzpicture}[scale=.7,every node/.style={draw,shape=circle,outer sep=2pt,inner sep=1pt,minimum
			size=.2cm}]
		
		\foreach \j in {1, ..., 3}{
			\node[fill=black]  (\j1) at (-1-\j,1) {};
			\node[draw=none] at (-1-\j, 1.5){\small{$z_{\j}$}};
		}
		\foreach \j in {1,...,3}{
			\node[fill=black]  (\j0) at (\j-2,0) {};
			\node[draw=none] at (\j-2, -0.5){\small{$v_{\j}$}};
		}
		
		\foreach \j in {1, ..., 4}{
			\node[fill=black]  (\j2) at (-1-\j,-1) {};
			\node[draw=none] at (-1-\j, -1.5){\small{$y_{\j}$}};
		}
		\foreach \j in {1,...,4}{
			\node[fill=black]  (\j3) at (\j+1,1) {};
			\node[draw=none] at (\j+1, 1.5){\small{$x_{\j}$}};
		}
		\foreach \j in {1,...,4}{
			\node[fill=black]  (\j4) at (\j+1,-1) {};
			\node[draw=none] at (\j+1, -1.5){\small{$u_{\j}$}};
		}
		
		
		\draw[thick] (31)--(11)--(10)--(30)--(13)--(43);
		\draw[thick] (42)--(12)--(10);
		\draw[thick] (30)--(14)--(44);
		
		
		\end{tikzpicture}

	\end{minipage}%
	\begin{minipage}[c]{0.5\textwidth}
		\centering
		\begin{tikzpicture}[scale=.7,every node/.style={draw,shape=circle,outer sep=2pt,inner sep=1pt,minimum
			size=.2cm}]
		
		\foreach \x in {0,...,9}
		\foreach \y in {0,...,9}
		{
			\vertex[opacity=0.25]  (\x\y) at (\x,\y) {};
		}
		\foreach \x in {0,...,8}
		\foreach \y in {0,...,8}
		{
			\pgfmathtruncatemacro{\a}{\x+1}
			\pgfmathtruncatemacro{\b}{\y+1}
			\path[opacity=0.25]
			(\x\y) edge (\a\b)
			(\x\b) edge (\a\y);
		}
		\foreach \x in {0,...,8}
		\foreach \y in {0,...,9}
		{
			\pgfmathtruncatemacro{\a}{\x+1}
			\path[opacity=0.25]
			(\x\y) edge (\a\y)
			(\y\x) edge (\y\a);
		}

		\foreach \j in {0,...,9}{
			\pgfmathtruncatemacro{\a}{9-\j}
			\node[fill=black]  (\j\a) at (\j,\a) {};
		}
		\foreach \j in {2,...,9}{
			\pgfmathtruncatemacro{\a}{10-\j}
			\node[fill=black]  (\j\a) at (\j,\a) {};
		}
		
		\draw[thick] (0,9) --(4,5);
		\draw[thick] (5,4)--(9,0);
		\draw[thick] (4,5)--(5,5) --(5,4);
		\draw[thick] (2,8)--(4,6) --(4,5);
		\draw[thick] (5,4) -- (6,4) -- (9,1);
		
		\draw [dashed] (1,8) -- (2,8) --(2,7)--(3,7) -- (3,6) --(4,6);
		\draw [dashed] (6,4) -- (6,3) -- (7,3) -- (7,2) -- (8,2) --(8,1)--(9,1) -- (9,0);
		\draw [dashed] (4,6) -- (6,4);
		\draw [dashed] (4,5) -- (5,4);
		
		\node[draw=none] at (-0.5, 9){\small{$y_4$}};
		\node[draw=none] at (0.5, 8){\small{$y_3$}};
		\node[draw=none] at (2.5, 8){\small{$z_3$}};
		\node[draw=none] at (1.5, 7){\small{$y_2$}};
		\node[draw=none] at (3.5, 7){\small{$z_2$}};
		\node[draw=none] at (2.5, 6){\small{$y_1$}};
		\node[draw=none] at (4.5, 6){\small{$z_1$}};
		\node[draw=none] at (3.5, 5){\small{$v_1$}};
		\node[draw=none] at (5.5, 5){\small{$v_2$}};
		\node[draw=none] at (4.5, 4){\small{$v_3$}};
		\node[draw=none] at (6.5, 4){\small{$x_1$}};
		
		\node[draw=none] at (5.5, 3){\small{$u_1$}};
		\node[draw=none] at (7.5, 3){\small{$x_2$}};
		\node[draw=none] at (6.5, 2){\small{$u_2$}};
		\node[draw=none] at (8.5, 2){\small{$x_3$}};
		\node[draw=none] at (7.5, 1){\small{$u_3$}};
		\node[draw=none] at (9.5, 1){\small{$x_4$}};
		\node[draw=none] at (8.5, 0){\small{$u_4$}};
		\end{tikzpicture}
	\end{minipage}
\caption{Illustrating the embedding of a tree  in $P^{\boxtimes, 2}_{10}$ \\  as described in Case 1 of Theorem \ref{strongdim3tree_has_threshold2}} \label{Case1}
\end{figure}

\bigskip

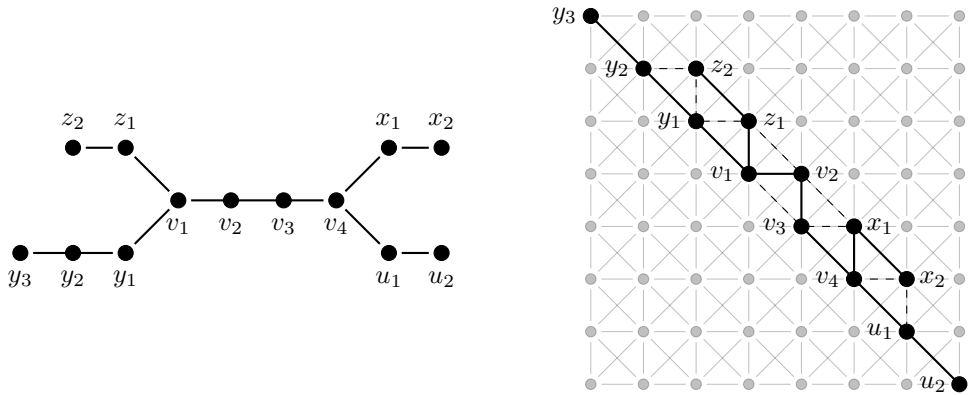
\begin{figure} [h] \label{Case2}
	\begin{minipage}[c]{0.5\textwidth}
		\centering
		\begin{tikzpicture}[scale=.7,every node/.style={draw,shape=circle,outer sep=2pt,inner sep=1pt,minimum
			size=.2cm}]
		
		\foreach \j in {1, ..., 2}{
			\node[fill=black]  (\j1) at (-1-\j,1) {};
			\node[draw=none] at (-1-\j, 1.5){\small{$z_{\j}$}};
		}
		\foreach \j in {1,...,4}{
			\node[fill=black]  (\j0) at (\j-2,0) {};
			\node[draw=none] at (\j-2, -0.5){\small{$v_{\j}$}};
		}
		
		\foreach \j in {1, ..., 3}{
			\node[fill=black]  (\j2) at (-1-\j,-1) {};
			\node[draw=none] at (-1-\j, -1.5){\small{$y_{\j}$}};
		}
		\foreach \j in {1,...,2}{
			\node[fill=black]  (\j3) at (\j+2,1) {};
			\node[draw=none] at (\j+2, 1.5){\small{$x_{\j}$}};
		}
		\foreach \j in {1,...,2}{
			\node[fill=black]  (\j4) at (\j+2,-1) {};
			\node[draw=none] at (\j+2, -1.5){\small{$u_{\j}$}};
		}
		
		\draw[thick] (21)--(11)--(10)--(40)--(13)--(23);
		\draw[thick] (32)--(12)--(10);
		\draw[thick] (40)--(14)--(24);
		
		\end{tikzpicture}
	\end{minipage}%
	\begin{minipage}[c]{0.5\textwidth}
		\centering
		\begin{tikzpicture}[scale=.7,every node/.style={draw,shape=circle,outer sep=2pt,inner sep=1pt,minimum
			size=.2cm}]
		
		\foreach \x in {0,...,7}
		\foreach \y in {0,...,7}
		{
			\vertex[opacity=0.25]  (\x\y) at (\x,\y) {};
		}
		\foreach \x in {0,...,6}
		\foreach \y in {0,...,6}
		{
			\pgfmathtruncatemacro{\a}{\x+1}
			\pgfmathtruncatemacro{\b}{\y+1}
			\path[opacity=0.25]
			(\x\y) edge (\a\b)
			(\x\b) edge (\a\y);
		}
		\foreach \x in {0,...,6}
		\foreach \y in {0,...,7}
		{
			\pgfmathtruncatemacro{\a}{\x+1}
			\path[opacity=0.25]
			(\x\y) edge (\a\y)
			(\y\x) edge (\y\a);
		}

		\foreach \j in {0,...,7}{
			\pgfmathtruncatemacro{\a}{7-\j}
			\node[fill=black]  (\j\a) at (\j,\a) {};
		}
		\foreach \j in {2,...,6}{
			\pgfmathtruncatemacro{\a}{8-\j}
			\node[fill=black]  (\j\a) at (\j,\a) {};
		}
		
		\draw[thick] (0,7) --(3,4);
		\draw[thick] (4,3)--(7,0);
		\draw[thick] (3,4)--(4,4) --(4,3);
		\draw[thick] (2,6)--(3,5) -- (3,4);
		\draw[thick] (5,2)--(5,3) -- (6,2);
		%
		\draw [dashed] (1,6) -- (2,6) --(2,5)--(3,5);
		\draw [dashed] (5,2)--(6,2) -- (6,1);
		\draw [dashed] (3,5) -- (5,3);
		\draw [dashed] (3,4) -- (4,3)--(5,3);
		
		\node[draw=none] at (-0.5, 7){\small{$y_3$}};
		\node[draw=none] at (0.5, 6){\small{$y_2$}};
		\node[draw=none] at (2.5, 6){\small{$z_2$}};
		\node[draw=none] at (1.5, 5){\small{$y_1$}};
		\node[draw=none] at (3.5, 5){\small{$z_1$}};
		\node[draw=none] at (2.5, 4){\small{$v_1$}};
		\node[draw=none] at (4.5, 4){\small{$v_2$}};
		\node[draw=none] at (3.5, 3){\small{$v_3$}};
		\node[draw=none] at (5.5, 3){\small{$x_1$}};
		\node[draw=none] at (4.5, 2){\small{$v_4$}};
		\node[draw=none] at (6.5, 2){\small{$x_2$}};
		\node[draw=none] at (5.5, 1){\small{$u_1$}};
		\node[draw=none] at (6.5, 0){\small{$u_2$}};
		\end{tikzpicture}	
	\end{minipage}	
\caption{Illustrating the embedding of a tree in $P^{\boxtimes, 2}_{8}$ \\ as described in Case 2 of Theorem \ref{strongdim3tree_has_threshold2}} \label{Case2}
\end{figure}

\begin{theorem}  \label{strongdim4tree_has_threshold2}
If $T$ is a tree with $\beta_s(T)=4$, then $\tau_s(T)=2$.
\end{theorem}

\begin{proof}

Let $T$ be a tree with strong dimension 4.  Then $T$ has exactly five end-vertices, see Figure \ref{strongdimension4}. Without loss of generality, we may assume $V(T)=\{t_1,t_2,\ldots, t_{k_6},$ $v_1,v_2,\ldots,v_{k_7},\ldots,v_{k_1},$ $u_1,u_2,\ldots,u_{k_2},$ $x_1,x_2,\ldots,x_{k_3},$ $y_1,y_2,\ldots,y_{k_4},$ $z_1,z_2,\ldots,z_{k_5}\}$ where $k_i \geq 1$ for $i=1,2,3,4,5,6,7$ and $k_7 \leq k_1$, and $E(T)=\{t_it_{i+1}:1\leq i \leq k_6-1\} \cup \{v_iv_{i+1}:1\leq i\leq k_1-1\}\cup \{u_iu_{i+1}:1 \leq i \leq k_2-1\}\cup \{x_ix_{i+1}:1 \leq i \leq k_3-1\}\cup \{y_iy_{i+1}:1 \leq i \leq k_4-1\}\cup \{z_iz_{i+1}:1 \leq i \leq k_5-1\} \cup \{v_1y_1,v_1z_1,v_{k_1}u_1,v_{k_1}x_1, v_{k_7}t_1\}$.   Thus by deleting the path $t_1,t_2, \ldots, t_{k_6}$ from $T$ we obtain a tree, call it $T'$, with four leaves that appears like the tree shown in Figure \ref{strongdim3}.

Depending on whether $k_1$ is odd or even, we now embed $T'$ in  $P^{\boxtimes, 2}_{D+1}$ (where $D$ is the diameter of $T$) as described in Case 1 or 2, respectively, of Theorem \ref{strongdim3tree_has_threshold2}. To complete the embedding of $T$, we now embed the path $t_1,t_2, \ldots, t_{k_6}$ in the SW-NE diagonal that passes through $v_{k_7}$ along the portion NE of $v_{k_7}$. Figures \ref{example1} and \ref{example2} illustrate these embeddings for specific trees falling into each of these two cases. The subgraph of $P_{D+1}^{\boxtimes, 2} $  induced by this embedding is a $\{y_{k_4}, u_{k_2}\}$-resolved embedding of $T$ that is also an isometric subgraph of $P_{D+1}^{ \boxtimes, 2} $. Thus $\tau_s(T) =2$.
\end{proof}

\begin{figure}[h] 
	\centering
	\begin{tikzpicture}[scale=.7,every node/.style={draw,shape=circle,outer sep=2pt,inner sep=1pt,minimum
		size=.2cm}]
	
	\foreach \j in {1, ..., 2}{
		\node[fill=black]  (\j1) at (-1-\j,1) {};
		\node[draw=none] at (-1-\j, 1.5){\small{$z_{\j}$}};
	}
	\foreach \j in {1,...,2}{
		\node[fill=black]  (\j0) at (\j-2,0) {};
		\node[draw=none] at (\j-2, -0.5){\small{$v_{\j}$}};
	}
	
	\foreach \j in {1, ..., 2}{
		\node[fill=black]  (\j2) at (-1-\j,-1) {};
		\node[draw=none] at (-1-\j, -1.5){\small{$y_{\j}$}};
	}
	\foreach \j in {1,...,2}{
		\node[fill=black]  (\j3) at (\j+4,1) {};
		\node[draw=none] at (\j+4, 1.5){\small{$x_{\j}$}};
	}
	\foreach \j in {1,...,2}{
		\node[fill=black]  (\j4) at (\j+4,-1) {};
		\node[draw=none] at (\j+4, -1.5){\small{$u_{\j}$}};
	}
	\foreach \j in {1,...,2}{
		\node[fill=black]  (\j5) at (2,\j) {};
		\node[draw=none] at (2.5, \j){\small{$t_{\j}$}};
	}
	\node[fill=black]  (40) at (2,0) {};
	\node[draw=none] at (2, -0.5){\small{$v_{k_7}$}};
	\node[fill=black]  (60) at (4,0) {};
	\node[draw=none] at (4, -0.5){\small{$v_{k_1}$}};
	\node[fill=black]  (41) at (-5,1) {};
	\node[draw=none] at (-5, 1.5){\small{$z_{k_5}$}};
	\node[fill=black]  (42) at (-5,-1) {};
	\node[draw=none] at (-5, -1.5){\small{$y_{k_5}$}};
	\node[fill=black]  (62) at (-7,-1) {};
	\node[draw=none] at (-7, -1.5){\small{$y_{k_4}$}};
	\node[fill=black]  (43) at (8,1) {};
	\node[draw=none] at (8, 1.5){\small{$x_{k_3}$}};
	\node[fill=black]  (44) at (8,-1) {};
	\node[draw=none] at (8, -1.5){\small{$u_{k_3}$}};
	\node[fill=black]  (64) at (10,-1) {};
	\node[draw=none] at (10, -1.5){\small{$u_{k_2}$}};
	\node[fill=black]  (45) at (2,4) {};
	\node[draw=none] at (2.5, 4){\small{$t_{k_6}$}};
	
	\draw[thick] (21)--(11)--(10)--(20);
	\draw[thick] (60)--(13)--(23);
	\draw[thick] (22)--(12)--(10);
	\draw[thick] (60)--(14)--(24);
	\draw[thick] (40) -- (25);
	
	\draw[thick] (20) -- (0.7, 0);
	\draw[thick, dotted] (0.7,0)--(1.3,0);
	\draw[thick] (1.3,0)--(40);
	
	\draw[thick] (40) -- (2.7, 0);
	\draw[thick, dotted] (2.7,0)--(3.3,0);
	\draw[thick] (3.3,0)--(60);
	\draw[thick] (41) -- (-4.3, 1);
	\draw[thick, dotted] (-4.3,1)--(-3.7,1);
	\draw[thick] (-3.7,1)--(21);
	\draw[thick] (42) -- (-4.3, -1);
	\draw[thick, dotted] (-4.3,-1)--(-3.7,-1);
	\draw[thick] (-3.7,-1)--(22);
	
	\draw[thick] (62) -- (-6.3, -1);
	\draw[thick, dotted] (-6.3,-1)--(-5.7,-1);
	\draw[thick] (-5.7,-1)--(42);
	
	\draw[thick] (23) -- (6.7, 1);
	\draw[thick, dotted] (6.7,1)--(7.4,1);
	\draw[thick] (7.4,1)--(43);
	
	\draw[thick] (24) -- (6.7, -1);
	\draw[thick, dotted] (6.7,-1)--(7.4,-1);
	\draw[thick] (7.4,-1)--(44);
	
	\draw[thick] (44) -- (8.7, -1);
	\draw[thick, dotted] (8.7,-1)--(9.4,-1);
	\draw[thick] (9.4,-1)--(64);
	
	\draw[thick] (25) -- (2,2.7);
	\draw[thick, dotted] (2,2.7)--(2,3.3);
	\draw[thick] (2,3.3)--(45);
	\end{tikzpicture}
\caption{A tree with strong dimension 4} \label{strongdimension4}
\end{figure}
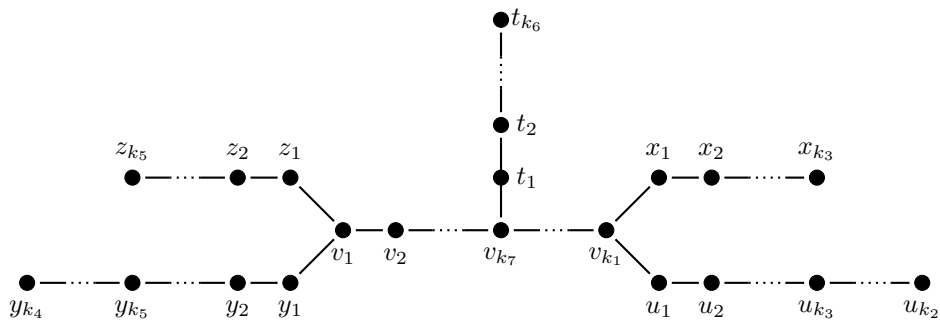 

\medskip

\begin{figure}[h]
	\begin{minipage}[c]{0.5\textwidth}
		\centering
		\begin{tikzpicture}[scale=.7,every node/.style={draw,shape=circle,outer sep=2pt,inner sep=1pt,minimum
			size=.2cm}]
		
		\foreach \j in {1, ..., 2}{
			\node[fill=black]  (\j1) at (-1-\j,1) {};
			\node[draw=none] at (-1-\j, 1.5){\small{$z_{\j}$}};
		}
		\foreach \j in {1,...,3}{
			\node[fill=black]  (\j0) at (\j-2,0) {};
			\node[draw=none] at (\j-2, -0.5){\small{$v_{\j}$}};
		}
		
		\foreach \j in {1, ..., 3}{
			\node[fill=black]  (\j2) at (-1-\j,-1) {};
			\node[draw=none] at (-1-\j, -1.5){\small{$y_{\j}$}};
		}
		\foreach \j in {1,...,2}{
			\node[fill=black]  (\j3) at (\j+1,1) {};
			\node[draw=none] at (\j+1, 1.5){\small{$x_{\j}$}};
		}
		\foreach \j in {1,...,2}{
			\node[fill=black]  (\j4) at (\j+1,-1) {};
			\node[draw=none] at (\j+1, -1.5){\small{$u_{\j}$}};
		}
		
		\foreach \j in {1,...,3}{
			\node[fill=black]  (\j5) at (0,\j) {};
			\node[draw=none] at (0.5, \j){\small{$t_{\j}$}};
		}
		
		\draw[thick] (21)--(11)--(10)--(30)--(13)--(23);
		\draw[thick] (32)--(12)--(10);
		\draw[thick] (30)--(14)--(24);
		\draw[thick] (20)--(35);	
		\end{tikzpicture}
	\end{minipage}%
	\begin{minipage}[c]{0.5\textwidth}
		\centering
		\begin{tikzpicture}[scale=.7,every node/.style={draw,shape=circle,outer sep=2pt,inner sep=1pt,minimum
			size=.2cm}]
		
		\foreach \x in {0,...,7}
		\foreach \y in {0,...,7}
		{
			\vertex[opacity=0.25]  (\x\y) at (\x,\y) {};
		}
		\foreach \x in {0,...,6}
		\foreach \y in {0,...,6}
		{
			\pgfmathtruncatemacro{\a}{\x+1}
			\pgfmathtruncatemacro{\b}{\y+1}
			\path[opacity=0.25]
			(\x\y) edge (\a\b)
			(\x\b) edge (\a\y);
		}
		\foreach \x in {0,...,6}
		\foreach \y in {0,...,7}
		{
			\pgfmathtruncatemacro{\a}{\x+1}
			\path[opacity=0.25]
			(\x\y) edge (\a\y)
			(\y\x) edge (\y\a);
		}
		
		\foreach \y in {0,...,7}
		{
			\draw[opacity=0.25]  (6.18,\y) -- (6.82,\y);
		}
		\foreach \y in {0,...,6}
		{
			\draw[opacity=0.25]  (7,\y+0.18) -- (7,\y+0.82);
			\draw[opacity=0.25]  (6.18,\y+0.18) -- (6.82,\y+0.82);
			\draw[opacity=0.25]	(6.18,\y+0.82) -- (6.82,\y+0.18);	
		}
		
		\foreach \j in {0,...,6}{
			\pgfmathtruncatemacro{\a}{6-\j}
			\node[fill=black]  (\j\a) at (\j,\a) {};
		}
		\foreach \j in {2,...,6}{
			\pgfmathtruncatemacro{\a}{7-\j}
			\node[fill=black]  (\j\a) at (\j,\a) {};
		}
		
		\foreach \j in {5,...,7}{
			\pgfmathtruncatemacro{\a}{\j-1}
			\node[fill=black]  (\j\a) at (\j,\a) {};
		}
		
		\draw[thick] (0,6) --(3,3);
		\draw[thick] (4,2)--(6,0);
		\draw[thick] (3,3)--(4,3) --(4,2)--(5,2)--(6,1);
		\draw[thick] (2,5)--(3,4) -- (3,3);
		\draw[thick] (4,3)--(7,6);
		\draw [dashed] (1,5) -- (2,5) --(2,4)--(3,4)--(5,2);
		\draw [dashed] (5,2)--(5,1) -- (6,1)--(6,0);
		\draw [dashed] (3,3) -- (4,2);

		\node[draw=none] at (-0.5, 6){\small{$y_3$}};
		\node[draw=none] at (0.5, 5){\small{$y_2$}};
		\node[draw=none] at (2.5, 5){\small{$z_2$}};
		\node[draw=none] at (1.5, 4){\small{$y_1$}};
		\node[draw=none] at (3.5, 4){\small{$z_1$}};
		\node[draw=none] at (2.5, 3){\small{$v_1$}};
		\node[draw=none] at (4.5, 3){\small{$v_2$}};
		\node[draw=none] at (3.5, 2){\small{$v_3$}};
		\node[draw=none] at (5.5, 2){\small{$x_1$}};
		\node[draw=none] at (4.5, 1){\small{$u_1$}};
		\node[draw=none] at (6.5, 1){\small{$x_2$}};
		\node[draw=none] at (5.5, 0){\small{$u_2$}};
		
		\node[draw=none] at (5.5, 4){\small{$t_1$}};
		\node[draw=none] at (6.5, 5){\small{$t_2$}};
		\node[draw=none] at (7.5, 6){\small{$t_3$}};
		\end{tikzpicture}
	\end{minipage}	
\caption{Illustrating an embedding of a tree in $P_8^{ \boxtimes, 2}$ \\
as described in Theorem \ref{strongdim4tree_has_threshold2} if $k_1$ is odd}
\label{example1}
\end{figure}
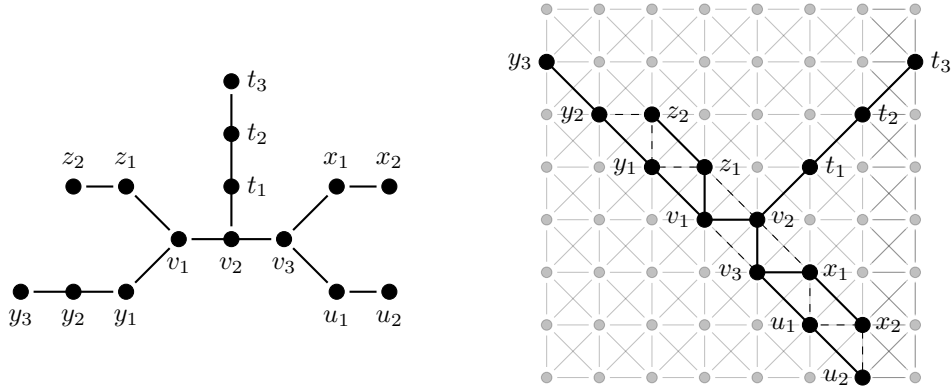

\medskip

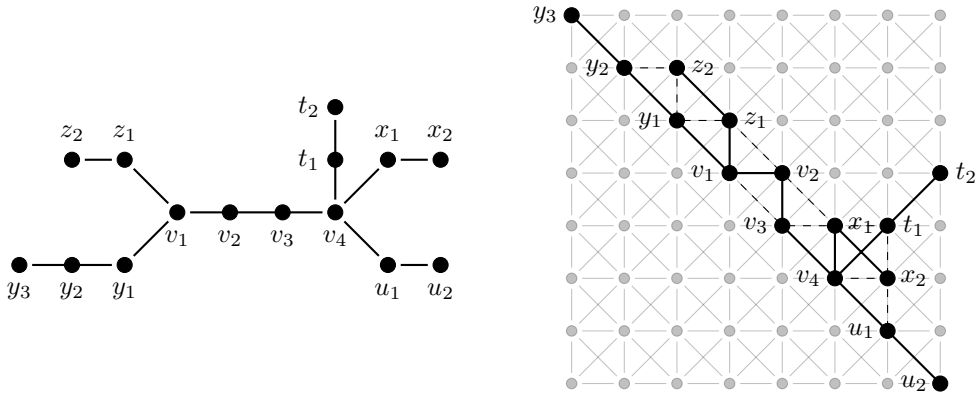
\begin{figure}[h]
		\begin{minipage}[c]{0.5\textwidth}
		\centering
		\begin{tikzpicture}[scale=.7,every node/.style={draw,shape=circle,outer sep=2pt,inner sep=1pt,minimum
			size=.2cm}]
		
		\foreach \j in {1, ..., 2}{
			\node[fill=black]  (\j1) at (-1-\j,1) {};
			\node[draw=none] at (-1-\j, 1.5){\small{$z_{\j}$}};
		}
		\foreach \j in {1,...,4}{
			\node[fill=black]  (\j0) at (\j-2,0) {};
			\node[draw=none] at (\j-2, -0.5){\small{$v_{\j}$}};
		}
		
		\foreach \j in {1, ..., 3}{
			\node[fill=black]  (\j2) at (-1-\j,-1) {};
			\node[draw=none] at (-1-\j, -1.5){\small{$y_{\j}$}};
		}
		\foreach \j in {1,...,2}{
			\node[fill=black]  (\j3) at (\j+2,1) {};
			\node[draw=none] at (\j+2, 1.5){\small{$x_{\j}$}};
		}
		\foreach \j in {1,...,2}{
			\node[fill=black]  (\j4) at (\j+2,-1) {};
			\node[draw=none] at (\j+2, -1.5){\small{$u_{\j}$}};
		}
		\foreach \j in {1,...,2}{
			\node[fill=black]  (\j5) at (2,\j) {};
			\node[draw=none] at (1.5, \j){\small{$t_{\j}$}};
		}
		
		\draw[thick] (21)--(11)--(10)--(40)--(13)--(23);
		\draw[thick] (32)--(12)--(10);
		\draw[thick] (25)--(40)--(14)--(24);
		
		\end{tikzpicture}
	\end{minipage}%
	\begin{minipage}[c]{0.5\textwidth}
		\centering
		\begin{tikzpicture}[scale=.7,every node/.style={draw,shape=circle,outer sep=2pt,inner sep=1pt,minimum
			size=.2cm}]
		
		\foreach \x in {0,...,7}
		\foreach \y in {0,...,7}
		{
			\vertex[opacity=0.25]  (\x\y) at (\x,\y) {};
		}
		\foreach \x in {0,...,6}
		\foreach \y in {0,...,6}
		{
			\pgfmathtruncatemacro{\a}{\x+1}
			\pgfmathtruncatemacro{\b}{\y+1}
			\path[opacity=0.25]
			(\x\y) edge (\a\b)
			(\x\b) edge (\a\y);
		}
		\foreach \x in {0,...,6}
		\foreach \y in {0,...,7}
		{
			\pgfmathtruncatemacro{\a}{\x+1}
			\path[opacity=0.25]
			(\x\y) edge (\a\y)
			(\y\x) edge (\y\a);
		}

		\foreach \j in {0,...,7}{
			\pgfmathtruncatemacro{\a}{7-\j}
			\node[fill=black]  (\j\a) at (\j,\a) {};
		}
		\foreach \j in {2,...,6}{
			\pgfmathtruncatemacro{\a}{8-\j}
			\node[fill=black]  (\j\a) at (\j,\a) {};
		}
		\foreach \j in {1,...,2}{
			\pgfmathtruncatemacro{\a}{\j+2}
			\node[fill=black]  (\j\a) at (\j+5,\a) {};
		}
		
		\draw[thick] (0,7) --(3,4);
		\draw[thick] (4,3)--(7,0);
		\draw[thick] (3,4)--(4,4) --(4,3);
		\draw[thick] (2,6)--(3,5) -- (3,4);
		\draw[thick] (7,4)--(5,2)--(5,3) -- (6,2);
		%
		\draw[dashed] (1,6) -- (2,6) --(2,5)--(3,5);
		\draw[dashed] (5,2)--(6,2) -- (6,1);
		\draw[dashed] (3,5) -- (5,3);
		\draw[dashed] (3,4) -- (4,3)--(5,3)--(6,3)--(6,2);
		
		\node[draw=none] at (-0.5, 7){\small{$y_3$}};
		\node[draw=none] at (0.5, 6){\small{$y_2$}};
		\node[draw=none] at (2.5, 6){\small{$z_2$}};
		\node[draw=none] at (1.5, 5){\small{$y_1$}};
		\node[draw=none] at (3.5, 5){\small{$z_1$}};
		\node[draw=none] at (2.5, 4){\small{$v_1$}};
		\node[draw=none] at (4.5, 4){\small{$v_2$}};
		\node[draw=none] at (3.5, 3){\small{$v_3$}};
		\node[draw=none] at (5.5, 3){\small{$x_1$}};
		\node[draw=none] at (4.5, 2){\small{$v_4$}};
		\node[draw=none] at (6.5, 2){\small{$x_2$}};
		\node[draw=none] at (5.5, 1){\small{$u_1$}};
		\node[draw=none] at (6.5, 0){\small{$u_2$}};
		\node[draw=none] at (6.5, 3){\small{$t_1$}};
		\node[draw=none] at (7.5, 4){\small{$t_2$}};
		\end{tikzpicture}
	\end{minipage}
\caption{Illustrating an embedding of a tree in $P_8^{ \boxtimes, 2}$ \\
as described in Theorem \ref{strongdim4tree_has_threshold2} if $k_1$ is even}
\label{example2}
\end{figure}

It was shown in \cite{MolMurphyOellermann2019} that there are trees of arbitrarily large metric dimension that have threshold dimension 2. In particular it was shown that if $L_{3n}$, for $n \ge 2$, is the tree obtained from a path $P:v_1v_2 \ldots v_n$ by attaching, for each $1 \le i \le n$, two leaves $u_i$ and $w_i$ to $v_i$, then the threshold dimension of $L_{3n}$ is $2$ while the metric dimension is $n$. This was shown by describing a $\{u_1,w_1\}$-resolved embedding in $P_{n+1}^{\boxtimes,2}$. The same embedding also shows that $\tau_s(L_{3n}) = 2$. We briefly describe this embedding here. We again assume that the vertices of $P_{n+1}$ have been labeled $0, 1, \ldots, n$ so that, for $0 \le i <n$, the vertex $i$ is adjacent with $i+1$. Embed the path $v_1v_2 \ldots v_n$ along the SW-NE diagonal passing through $(0,0)$ and $(n,n)$ by placing $v_1$ in position $(1,1)$ and the remaining vertices of $P$ in order along the diagonal NE of $(1,1)$. Next the vertices $u_1, \ldots, u_n$ are placed in order along the SW-NE diagonal through $(0,1)$ and $(n-1,n)$ starting with $u_1$ in position $(0,1)$. Finally the vertices $w_1, w_2, \ldots, w_n$ are placed in order along the SW-NE diagonal through $(1,0)$ and $(n,n-1)$ starting with $w_1$ in position $(1,0)$. The subgraph induced by this embedding is a $\{u_1,w_1\}$-resolved embedding of $L_{3n}$  in $P_{n+1} ^{\boxtimes,2}$ that is an induced subgraph of $P_{n+1}^{ \boxtimes,2} $. Figure \ref{L_12} shows such an embedding for $L_{12}$.

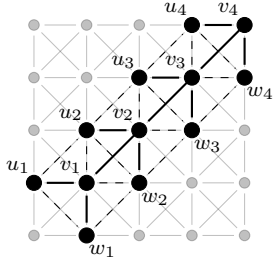
\begin{figure}[t]

		\centering
		
		\begin{tikzpicture}[scale=.7,every node/.style={draw,shape=circle,outer sep=2pt,inner sep=1pt,minimum
	size=.2cm}]

\foreach \x in {0,...,4}
\foreach \y in {0,...,4}
{
	\vertex[opacity=0.25]  (\x\y) at (\x,\y) {};
}
\foreach \x in {0,...,3}
\foreach \y in {0,...,3}
{
	\pgfmathtruncatemacro{\a}{\x+1}
	\pgfmathtruncatemacro{\b}{\y+1}
	\path[opacity=0.25]
	(\x\y) edge (\a\b)
	(\x\b) edge (\a\y);
}
\foreach \x in {0,...,3}
\foreach \y in {0,...,4}
{
	\pgfmathtruncatemacro{\a}{\x+1}
	\path[opacity=0.25]
	(\x\y) edge (\a\y)
	(\y\x) edge (\y\a);
}

\foreach \j in {1,...,4}{
	\pgfmathtruncatemacro{\a}{\j-1}
	\node[fill=black]  (\j\j) at (\j,\j) {};
	\node[fill=black]  (\a\j) at (\a,\j) {};
	\node[fill=black]  (\j\a) at (\j,\a) {};	
}

\foreach \j in {1,...,4}{
	\pgfmathtruncatemacro{\a}{\j-1}
	\pgfmathtruncatemacro{\b}{\j+1}
	\draw[thick] (\a\j) --(\j\j);
	\draw[thick] (\j\a) --(\j\j);
	\draw[dashed] (\a\j)--(\j\a);
}

\foreach \j in {1,...,3}{
	\pgfmathtruncatemacro{\a}{\j-1}
	\pgfmathtruncatemacro{\b}{\j+1}
	\draw[thick] (\j\j) --(\b,\b);
	\draw[dashed] (\a\j) --(\j\b);
	\draw[dashed] (\j\a)--(\b\j);
	\draw[dashed] (\j\j)--(\b\j);
	\draw[dashed] (\j\j)--(\j\b);
}

\node[draw=none]  at (-0.3,1.3) {\scriptsize $u_1$};
\node[draw=none]  at (0.7,2.3) {\scriptsize $u_2$};
\node[draw=none]  at (1.7,3.3) {\scriptsize $u_3$};
\node[draw=none]  at (2.7,4.3) {\scriptsize $u_4$};

\node[draw=none]  at (0.7, 1.3){\scriptsize$v_1$};
\node[draw=none]  at(1.7, 2.3){\scriptsize$v_2$};
\node[draw=none]  at (2.7, 3.3){\scriptsize$v_3$};
\node[draw=none]  at (3.7, 4.3){\scriptsize$v_4$};

\node[draw=none]  at (1.3, -0.3){\scriptsize$w_1$};
\node[draw=none]  at (2.3, 0.7){\scriptsize$w_2$};
\node[draw=none]  at (3.3, 1.7){\scriptsize$w_3$};
\node[draw=none]  at (4.3, 2.7){\scriptsize$w_4$};

\end{tikzpicture}

	\caption{A $\{u_1,w_1\}$-resolved embedding $\varphi$ of the graph $L_{12}$ in $P_{5}^{\boxtimes,2}$.}
	\label{L_12}
\end{figure}

\section{Concluding Remarks}

In this paper we introduced the threshold strong dimension of a graph. We established an expression for the threshold strong dimension of a graph in terms of a minimum number of paths, each of sufficiently large order, whose strong product admits a certain type of embedding of the graph. We used this embedding result to show that there are graphs whose threshold dimension does not equal the threshold strong dimension. This embedding  result also led to the main idea for determining all graphs with vertex covering number 2 that can be realized as the strong resolving graph of a graph and it was used to show that all trees with strong dimension 3 or 4 have threshold strong dimension 2. For graphs in general we established sharp upper bounds for the threshold strong dimension.

We did not consider the computational complexity of finding the threshold strong dimension of a graph. In particular it is not known whether the following problems are NP-complete:

\medskip

\noindent {\bf Problem 1} For a given graph $G$ and positive integer $b$, does there exist $H \in \mathcal{U}(G)$ and a set $B \subseteq V(G)$ of cardinality $b$ such that $B$ strongly resolves $H$?

\medskip

\noindent {\bf Problem 2} Is Problem 1 NP-complete even  if we restrict ourselves to the class of trees?

\medskip
Recall that a graph $G$ is $\beta_s$-\emph{irreducible} if $\beta_s(G)=\tau_s(G)$.  The paths are precisely the graphs with strong dimension 1 that are $\beta_s$-irreducible. As remarked before, all graphs of strong dimension 2 are also $\beta_s$-irreducible.  The complete graphs of order $n$ are the graphs with strong dimension $n-1$ that are $\beta_s$-irreducible.  However, the following problem remains open.  

\medskip

\noindent {\bf Problem 3} For a given $k$, $3 \leq k \leq n-2$, characterize all graphs of order $n$ and strong dimension $k$ that are $\beta_s$-irreducible.

\bigskip

\bigskip

\section*{Author Contact Information}
\begin{enumerate}
\item Nadia Benakli, Department of Mathematics, New York City College of Technology, 300 Jay Street, Brooklyn, NY 11201, USA 
\item Novi H. Bong, Department of Mathematical Sciences, University of Delaware, 15 Orchard Road, Newark, DE, 19716, USA.
\item Shonda Dueck (Gosselin), Department of Mathematics and Statistics, The University of Winnipeg, 515 Portage Ave., Winnipeg, MB R3B 2E9, CANADA
\item Linda Eroh, Department of Mathematics, University of Wisconsin Oshkosh, 800 Algoma Boulevard, Oshkosh, WI 54963, USA.
\item Beth Novick, School of Mathematical and Statistical Sciences, Clemson University, O-110 Martin Hall, Box 340975, Clemson,  SC, 29634,  USA.
\item Ortrud R. Oellermann, Department of Mathematics and Statistics, The University of Winnipeg, 515 Portage Ave., Winnipeg, MB R3B 2E9, CANADA
\end{enumerate}

\section*{Acknowledgements}
We thank the Institute for Mathematics and its Applications at the University of Minnesota for hosting and sponsoring the Workshop for Women in Graph Theory and Applications, August 18-23, 2019, where the research for this project was started. \\

\noindent The last author is supported by an NSERC Grant CANADA, Grant number RGPIN-2016-05237.

\end{document}